\documentclass[parskip=never,abstracton]{scrartcl}

\usepackage{csquotes}

\usepackage[T1]{fontenc}
\usepackage[utf8]{inputenc}
\usepackage{lmodern}
\usepackage{textcomp}
\usepackage{setspace}
\usepackage{titlesec}
\usepackage[a4paper,left=2.5cm,right=2cm,top=2.5cm,bottom=2.5cm]{geometry}

\usepackage{etoolbox}
\usepackage{xparse}

\usepackage{mathtools}
\usepackage{amsthm,thmtools}
\usepackage{dsfont}
\let\mathbb\mathds
\usepackage{amssymb}
\usepackage[scr]{rsfso} 
\usepackage{bm}
\usepackage{euscript}
\usepackage{amsbsy}
\DeclareMathAlphabet\mathbfcal{OMS}{cmsy}{b}{n}

\usepackage{microtype}
\usepackage{authblk}
\usepackage{graphicx} 
\usepackage{tikz}
\usetikzlibrary{cd,quotes}
\usetikzlibrary{decorations.markings}
\usepackage{pgfplots}
\pgfplotsset{compat=1.13}
\usepackage{booktabs}
\usepackage{float}
\usepackage{environ}
\usepackage{xcolor,soul}
\usepackage[shortlabels]{enumitem}
\usepackage{eso-pic}

\usepackage{hyperref} 
\hypersetup{%
	colorlinks,%
	linkcolor={red!60!black},%
	citecolor={blue!50!black},%
	urlcolor={blue!80!black}%
}

\usepackage{tikz-cd}
\tikzset{Rightarrow/.style={double equal sign distance,>={Implies},->},
	triple/.style={-,preaction={draw,Rightarrow}},
	quadruple/.style={preaction={draw,Rightarrow,shorten >=0pt},shorten >=1pt,-,double,double
		distance=0.2pt}}

\def\on{\operatorname}
\def\ho{\on{ho}}
\def\CC{\mathbb{C}}

\def\DD{\mathbb{D}}
\def\OO{\mathbb{O}}
\def\Fun{\on{Fun}}
\def\Cat{\on{Cat}}

\def\Hom{\on{Hom}}
\def\D{\EuScript{D}}
\def\Nerv{\on{N}}
\def\NN{\on{N_2}}
\def\Lcat{{\on{LCat}}}
\def\id{\on{id}}
\def\Set{\on{Set}}

\SetEnumitemKey{axioms}{itemsep=0pt,align=left,itemindent=!,
                        listparindent=\parindent,parsep=\parskip,
                        before=,
                        label={\Roman*}}

\SetEnumitemKey{claims}{itemsep=0pt,align=left,itemindent=!,
                        listparindent=\parindent,parsep=\parskip,
                        before=,
                        label={\Roman*}}

\DeclarePairedDelimiterX\set[1]{\lbrace}{\rbrace}{#1}
\newlist{implications}{description}{1} 
\setlist[implications]{itemsep=0pt,leftmargin=\parindent}
\NewDocumentCommand\implication{o}
  {\IfValueTF{#1}
    {\auximplication#1\relax}
    {\item[\normalfont($\,\Rightarrow\,$)]}}
\NewDocumentCommand\auximplication{u-u\relax}
  {\item[\normalfont(#1)$\,\Rightarrow\,$(#2)]}
\newcommand*\backimplication{\item[\normalfont($\,\Leftarrow\,$)]}

\makeatletter
\newcounter{diagram}[section]
\def\thediagram{\thesection.\arabic{diagram}}
\def\ftype@diagram{4}
\def\ext@diagram{diag}
\def\fnum@diagram{Diagram~\thediagram}
\def\fs@diagram{htbp!}
\ExplSyntaxOn
\NewDocumentEnvironment{diagram}{O{htbp!}m}
  {\@float{diagram}[#1]\centering}
  {
   
   \caption{}
   \label{#2}
   \end@float
  }

\newcounter{subdiagram}[diagram]
\def\thesubdiagram{\thediagram.\arabic{subdiagram}}
\NewEnviron{multidiagram}{\expandafter\domultidiagram\BODY\enddomultidiagram}
\NewDocumentCommand\domultidiagram{omu\enddomultidiagram}
 {
  \IfValueTF{#1}{\diagram[#1]}{\diagram}{}
   \refstepcounter{diagram}
   \centering
    \seq_clear:N \l_tmpb_seq
    \seq_set_split:Nnn \l_tmpa_seq { \next } { #3 }
    \seq_map_inline:Nn \l_tmpa_seq
     {
      \seq_put_right:Nn \l_tmpb_seq
       {
        \begin{tabular}[b]{@{}c@{}}
         ##1 \\[3ex]
         \refstepcounter{subdiagram}
         \label{#2\othercolon\the\value{subdiagram}}
         Diagram~\thesubdiagram 
        \end{tabular}
       }
     }
    \seq_use:Nn \l_tmpb_seq { \qquad }
   \let\label\@gobble
   \let\caption\@gobble
  \enddiagram
 }
\ExplSyntaxOff
\def\othercolon{:}
\makeatother

\declaretheoremstyle[bodyfont=\itshape,notefont=\bfseries]{abellanA}
\declaretheoremstyle[notefont=\bfseries]{abellanB}

\declaretheorem[style=abellanA,numberwithin=subsection,name={Theorem}]{theorem}
\declaretheorem[style=abellanA,numberwithin=subsection,name={Conjecture}]{conjecture}
\declaretheorem[style=abellanA,numberlike=theorem,name={Lemma}]{lemma}
\declaretheorem[style=abellanB,numberlike=theorem,name={Definition}]{definition}
\declaretheorem[style=abellanB,numberlike=theorem,name={Remark}]{remark}
\declaretheorem[style=abellanB,numberlike=theorem,name={Construction}]{construction}
\declaretheorem[style=abellanA,numberlike=theorem,name={Proposition}]{proposition}
\declaretheorem[style=abellanB,numberlike=theorem,name={Example}]{example}
\declaretheorem[style=abellanB,numbered=no,name={Notation}]{notation}
\declaretheorem[style=abellanA,numberlike=theorem,name={Corollary}]{corollary}

\docsvlist{theorem,lemma,definition,remark,proposition,example,corollary}

\let\leq\leqslant
\let\geq\geqslant
\let\epsilon\varepsilon
\let\isom\cong

\newcommand*\mathblank{\mathord{-}}

\def\msSet{{\on{Set}_{\Delta}^+}}

\def\Cat{\on{Cat}}

\usepackage[bbgreekl]{mathbbol}
\def\cchi{\mathbb{\bbchi}}
\def\PPhi{\mathbb{\bbphi}}

\let\emptyset\varnothing

\makeatletter
\newcommand{\fixed@sra}{$\vrule height 2\fontdimen22\textfont2 width 0pt\rightarrow$}
\newcommand{\shortarrowup}[1]{%
	\mathrel{\text{\rotatebox[origin=c]{65}{\fixed@sra}}}
}
\newcommand{\shortarrowdown}[1]{%
	\mathrel{\text{\rotatebox[origin=c]{250}{\fixed@sra}}}
}
\makeatother

\newcommand{\upslash}{\!\shortarrowup{1}}
\newcommand{\downslash}{\!\shortarrowdown{1}}

\def\lra{\longrightarrow}
\def\lla{\longleftarrow}

\def\llra{\def\arraystretch{.1}\begin{array}{c} \lra \\ \lla \end{array}}

\tikzcdset{
  arrows={},
  nat/.style={Rightarrow},
  nat1/.style={nat,shift left=.8ex},
  nat2/.style={nat,shift right=.8ex,swap},
  map/.style={},
  map1/.style={map,shift left=.7ex},
  map2/.style={map,shift right=.7ex,swap},
  mapA/.style={map,shift left=.5ex},
  mapB/.style={map,shift left=.5ex},
  unique/.style={line cap=round,densely dotted},
  incl/.style={hookrightarrow},
  mono/.style={rightarrowtail},
  epi/.style={twoheadrightarrow},
  iso/.style={map,sloped,"{\tikz[x=.6ex,y=.3ex]{\draw[-](0,0)sin(1,1)cos(2,0)sin(3,-1)cos(4,0);}}"},
  line/.style={dash},
  line1/.style={line,shift left=.3ex},
  line2/.style={line,shift right=.3ex,swap},
}

\DeclareMathOperator\Nsc{N^{sc}}
\def\op{{\on{op}}}

\makeatletter
\newcommand*\dirlim{\mathop{\mathpalette\varlim@{\rightarrowfill@\scriptscriptstyle}}\nmlimits@}
\newcommand*\prolim{\mathop{\mathpalette\varlim@{\leftarrowfill@\scriptscriptstyle}}\nmlimits@}
\makeatother

\def\llra{\def\arraystretch{.1}\begin{array}{c} \lra \\ \lla \end{array}}

\newcommand{\nat}{\Rightarrow}

\tikzset{
  abellanarrows/.style={line cap=round,line join=round,line width=.4pt},
  abellanarrowlength/.store in=\abellanarrowlength,
}

\ExplSyntaxOn

\cs_generate_variant:Nn \tl_replace_all:Nnn { Nf }

\NewDocumentCommand \func { s O{} m }
 {
  \group_begin:
   \IfBooleanTF{#1}
    { \keys_set:nn { abellan / func } { aligned = true , #2 } }
    { \keys_set:nn { abellan / func } {#2} }
   \abellan_func:n {#3}
  \group_end:
 }
\NewDocumentCommand \arr { s o m }
 {
  \IfBooleanF{#1}
   { \bool_if:NT \l_abellan_aligned_bool { & } }
  \abellan_arr:n {#3}
 }
\NewDocumentCommand \addarr { o m m }
 {
  \keys_set:nn { abellan / func / addarrow } { name = {#2} , #3 }
  \tl_clear:N \l_abellan_arrname_tl
 }
\NewDocumentCommand \setupfunc { m } { \keys_set:nn { abellan / func } {#1} }

\bool_new:N \l_abellan_aligned_bool
\tl_new:N \l_abellan_func_tl
\tl_new:N \l_abellan_arrname_tl
\tl_new:N \l_abellan_arrmode_tl
\tl_new:N \g_abellan_func_substitutions_tl
\quark_new:N \q_abellan

\keys_define:nn { abellan / func }
 {
  aligned .bool_set:N = \l_abellan_aligned_bool ,
  aligned .initial:n = false ,
  mode .choices:nn = { base , tikz } { \tl_set_eq:NN \l_abellan_arrmode_tl \l_keys_choice_tl } ,
  mode .initial:n = base ,
  tikzarrowlength .code:n = \tikzset{abellanarrowlength={#1}} ,
  tikzarrowlength .initial:n = 2em ,
  tikzarrowstyle .code:n = \tikzset{abellanarrows/.append ~ style={#1}} ,
 }
\keys_define:nn { abellan / func / addarrow }
 {
  name .tl_set:N = \l_abellan_arrname_tl ,
  base .code:n =
    \abellan_addarrow:nnww { \l_abellan_arrname_tl } { base } #1 \q_abellan ,
  tikz .code:n =
    \abellan_addarrow:nnww { \l_abellan_arrname_tl } { tikz } #1 \q_abellan ,
  substitutions .code:n =
   \tl_map_inline:nn {#1}
    {
     \tl_gput_right:Nx \g_abellan_func_substitutions_tl
      {
       \exp_not:n { \tl_replace_all:Nnn \l_abellan_func_tl {##1} }
        { \arr{\l_abellan_arrname_tl} }
      }
    } ,
 }
\cs_new_protected:Npn \abellan_func:n #1
 {
  \resetdynamicto
  \tl_set:Nn \l_abellan_func_tl {#1}
  \tl_replace_all:Nfn \l_abellan_func_tl
   { \char_generate:nn { `: } { 12 } } { \colon }
  \bool_if:NTF \l_abellan_aligned_bool
   { \tl_replace_all:Nnn \l_abellan_func_tl { ; } { \\ } }
   { \tl_replace_all:Nnn \l_abellan_func_tl { ; } { ,\ } }
  \tl_use:N \g_abellan_func_substitutions_tl
  \bool_if:NT \l_abellan_aligned_bool { \! \begin {aligned} \scan_stop: }
  \tl_use:N \l_abellan_func_tl
  \bool_if:NT \l_abellan_aligned_bool { \\ \end {aligned} }
 }
\NewDocumentCommand \abellan_addarrow:nnww { m m O{} u\q_abellan }
 {
  \exp_args:Nc \NewDocumentCommand { abellan_arr_#1_#2:w } { #3 }
   {
    \use:c { abellan_arr_ \l_abellan_arrmode_tl :n } { #4 }
   }
 }
\cs_new_protected:Npn \abellan_arr:n #1
 {
  \use:c { abellan_arr_#1_\l_abellan_arrmode_tl :w }
 }
\cs_new_protected:Npn \abellan_arr_base:n #1
 {
  \mathrel{#1}
 }
\cs_new_protected:Npn \abellan_arr_tikz:n #1
 {
  \mathrel{\vcenter{\hbox{\tikz[abellanarrows]{#1}}}}
 }
\ExplSyntaxOff

\addarr{monomorphism}
 {
  substitutions = {>->}{ mono }{\rightarrowtail}{\mono} ,
  base = [o]{\IfValueTF{#1}{\rightarrowtail{#1}}{\mono}} ,
  tikz = [o]{\draw[>->,inner sep=0pt]
     (0,0)--({max(\abellanarrowlength,.5em\IfValueT{#1}{+width("$\scriptstyle#1$")})},0)
     \IfValueT{#1}{node[midway,above=.7ex]{\smash{$\scriptstyle#1$}}
     node[midway,below=.7ex]{}}
     ;} ,
  }
\addarr{mapsto}
 {
  substitutions = {\mapsto}{ mapsto } ,
  base = \mapsto ,
  tikz = {\draw[|->](0,0)--(\abellanarrowlength,0);} ,
 }
\addarr{rrarrows} 
 {
  substitutions = {\rightrightarrows}{ doublerr }{\doublerr} ,
  base = \rightrightarrows ,
  tikz = {\draw[>={To[width=.8ex,length=.4ex]},->](0,0)--(\abellanarrowlength,0);
          \draw[yshift=.8ex,>={To[width=.8ex,length=.4ex]},->](0,0)--(\abellanarrowlength,0);} ,
 }
\addarr{inclusion}
 {
  substitutions = {\hookrightarrow}{\incl}{ incl } ,
  base = [o]{\IfValueTF{#1}{\hookrightarrow{#1}}{\incl}} ,
  tikz = {\draw[{Hooks[right]}->](0,0)--(\abellanarrowlength,0);} ,
 }
\addarr{natural}
 {
 substitutions = {\Rightarrow}{=>}{ nat }{\nat} ,
  base = [o]{\IfValueTF{#1}{\xRightarrow{#1}}{\nat}} ,
  tikz = [o]{\draw[line cap=butt,double equal sign distance,-Implies,inner sep=0pt]
    (0,0)--({max(\abellanarrowlength,.3em\IfValueT{#1}{+width("$\scriptstyle#1$")})},0)
    \IfValueT{#1}{node[midway,above=.7ex]{\smash{$\scriptstyle#1$}}
    node[midway,below=.7ex]{}}
    ;} ,
 }
\addarr{epimorphism} 
 {
  substitutions = {\twoheadrightarrow}{->>}{ epi }{\epi},
   base = [o]{\IfValueTF{#1}{\xtwoheadrightarrow{#1}}{\epi}} ,
   tikz = [o]{\draw[->>,inner sep=0pt]
      (0,0)--({max(\abellanarrowlength,.5em\IfValueT{#1}{+width("$\scriptstyle#1$")})},0)
      \IfValueT{#1}{node[midway,above=.7ex]{\smash{$\scriptstyle#1$}}
      node[midway,below=.7ex]{}}
      ;} ,
   }
\addarr{to}
 {
  substitutions = {\to}{\rightarrow}{ to }{ funct }{ map }{->} ,
  base = [o]{\IfValueTF{#1}{\xrightarrow{#1}}{\to}} ,
  tikz = [o]{\draw[->,inner sep=0pt]
    (0,0)--({max(\abellanarrowlength,.5em\IfValueT{#1}{+width("$\scriptstyle#1$")})},0)
    \IfValueT{#1}{node[midway,above=.7ex]{\smash{$\scriptstyle#1$}}
    node[midway,below=.7ex]{}}
    ;} ,
 }
\newcommand*\resetdynamicto
  {\gdef\dynamicto{\arr*{to}\gdef\dynamicto{\arr*{mapsto}}}}
\resetdynamicto

\setupfunc{mode=tikz}

\ExplSyntaxOn
\NewDocumentCommand \printheader { m o m }
 {
  \par\noindent
  \begin{minipage}[t]{\textwidth}\noindent
  
  \begin{tabular}[t]{ll}
     & \keyval_parse:NNn \abellan_printname:n \abellan_printnamemail:nn { #3 } 
  \end{tabular}
  \vspace{.4cm}
  \end{minipage}
  \begin{center}\Large\bfseries
   #1 \IfValueT{#2}{\\[1ex] \large #2}
  \end{center}
  \vspace{.6cm}
 }
\tl_new:N \g_abellan_autores_tl
\tl_new:N \g_abellan_asignatura_tl
\cs_new_protected:Npn \abellan_printnamemail:nn #1 #2
 {
  #1 \\ & \quad\texttt{#2} \\ &
 }
\cs_new_protected:Npn \abellan_printname:n #1
 {
  #1 \\ &
 }
\ExplSyntaxOff

\makeatletter
\tikzcdset{
	open/.code     = {\tikzcdset{hook, circled};},
	open'/.code    = {\tikzcdset{hook', circled};},
	circled/.code  = {\tikzcdset{markwith = {\draw (0,0) circle (.375ex);}};},
	markwith/.code ={
		\pgfutil@ifundefined%
		{tikz@library@decorations.markings@loaded}%
		{\pgfutil@packageerror{tikz-cd}{You need to say %
				\string\usetikzlibrary{decorations.markings} to use arrows with markings}{}}{}%
		\pgfkeysalso{/tikz/postaction = {
				/tikz/decorate,
				/tikz/decoration={markings, mark = at position 0.5 with {#1}}}
		}
	},
}
\makeatother

\usepackage{glossaries}

\newglossaryentry{CCopop}{
	name={$\CC^{(-,\op)}$},
	description={The 2-morphism dual}
}

\newglossaryentry{N2}{
	name={$\Nerv_2$},
	description={The Duskin 2-nerve}
}

\newglossaryentry{OI}{
	name={$\OO^I$},
	description={The free 2-category on the $I$-simplex}
}

\newglossaryentry{lslice}{
	name={$\CC_{c /}$},
	description={{The lax slice category of a 2-category}}
}

\newglossaryentry{LW}{
	name={$L_{\mathcal{W}}$},
	description={The $(\infty,1)$-categorical localization functor}}

\newglossaryentry{Cdag}{
	name={$\CC^\dagger$},
	description={A marked 2-category}
}

\newglossaryentry{cchi}{
	name={$\cchi_{\CC}$},
	description={The relative 2-nerve}
}

\newglossaryentry{tilcchi}{
	name={$\widetilde{\cchi}_{\CC}$},
	description={The `thickened' relative 2-nerve}
}

\newglossaryentry{CCslash}{
	name={$\mathfrak{C}_{\CC /}$},
	description={Functors which encode the 2-functoriality of lax slices}
}

\newglossaryentry{colimdag}{
	name={$\on{colim}^\dagger F$},
	description={The marked colimit of a 2-functor $F$}
}

\newglossaryentry{El}{
	name={$\on{El}(F)^\dagger$},
	description={The 2-categorical Grothendieck construction}
}

\renewcommand{\glossarysection}[2][]{}

\makeglossaries

\def\scr{\EuScript}

\title{Theorem A for marked 2-categories}
\author{Fernando Abell\'an Garc\'ia, Walker H. Stern}
\date{}

\begin{document}
  \maketitle
	
		\begin{abstract}
			In this work, we prove a generalization of Quillen's Theorem A to 2-categories equipped with a special set of morphisms which we think of as weak equivalences, providing sufficient conditions for a 2-functor to induce an equivalence on $(\infty,1)$-localizations. When restricted to 1-categories with all morphisms marked, our theorem retrieves the classical Theorem A of Quillen. We additionally state and provide evidence for a new conjecture: the \emph{cofinality conjecture}, which describes the relation between a conjectural theory of marked $(\infty,2)$-colimits and our generalization of Theorem A. 
		\end{abstract}
	
\tableofcontents

\section*{Introduction}
\addcontentsline{toc}{section}{Introduction}
\subsection*{Towards a generalization of Theorem A}
\addcontentsline{toc}{subsection}{Towards a generalization of Theorem A}
Quillen's Theorems A and B are bedrock results in higher category theory, establishing conditions under which a functor of categories $\func{F:\mathcal{C}\to \mathcal{D}}$ defines a homotopy equivalence or fiber sequence of spaces, respectively. These theorems have been key to the modern understanding of algebraic topology --- not only in the original K-theoretic context of Quillen (cf. \cite{QuillenK}) but also in contexts ranging from algebraic topology to higher category theory. Philosophically speaking, since every homotopy type can be represented as the nerve of a category, Theorem A can be regarded as a fundamental tool for attacking homotopy-theoretic questions with explicit combinatorial presentations.

The criterion of Quillen's Theorem A --- that the slice categories of $F$ be contractible --- has additional significance in the study of homotopy (i.e. $(\infty,1)$-) colimits. A functor satisfies the criterion if and only if it is \emph{homotopy cofinal}, that is, restricting diagrams along it preserves their homotopy colimits. Particularly in the study of $\infty$-categorical Kan extensions, this makes the criterion of Theorem A an invaluable tool for explicit computation. 

In both of these areas, however, there are settings of interest in which Theorem A does not capture all of the salient features. Most notably, one often wants to associate an $\infty$-category to a 1-category, rather than simply a classifying space. In contexts where the aim is to retain more $\infty$-categorical structure, a generalization of Theorem A is thus highly desirable. 

With the benefit of the modern toolbox, one can make an observation that reframes the criterion of Theorem A: a non-empty Kan complex $K$ is contractible if and only if every object of $K$ is initial (when $K$ is viewed as an $\infty$-groupoid). One can then rephrase Quillen's Theorem A as: 

\begin{theorem}[Quillen's Theorem A]
	Let $F:\mathcal{C}\to \mathcal{D}$ be a functor between $1$-categories. For $d\in \mathcal{D}$, denote by $N(\mathcal{C}_{d/})^\simeq$ the $\infty$-groupoid completion of the overcategory. Suppose that, for every $d\in \mathcal{D}$, 
	\begin{itemize}
		\item There exists $c\in\mathcal{C}$ and a morphism $d\to F(c)$
		\item Every morphism $d\to F(b)$ represents an initial object of $N(\mathcal{C}_{d/})^\simeq$.
	\end{itemize}
	Then $\func{|F|:|N(\mathcal{C})|\to |N(\mathcal{D})|}$ is a homotopy equivalence. 
\end{theorem}

With this reframing, it becomes possible to generalize Theorem A to a much broader context. The $\infty$-groupoidification of $N(\mathcal{C})$ is nothing more or less than the $\infty$-categorical localization of $N(\mathcal{C})$ at all morphisms, so by generalizing from the set of all morphisms to any wide subcategory of weak equivalences, we can attempt to provide criteria under which a functor of $1$-categories with weak equivalences induces an equivalence on $\infty$-categorical localizations. It will turn out that the eventual form of our generalization will closely mirror the criteria above. 

Quillen's Theorem A relates several levels of categorification and strictness. It is a statement about \emph{strict 1-categories} which then implies a conclusion about \emph{$(\infty,0)$-groupoids}. Analogously, we aim to establish a statement about \emph{strict 2-categories} which implies a conclusion about \emph{$(\infty,1)$-categories}. To this end, we equip 2-categories with distinguished collections of 1-morphisms, yielding a notion we call \emph{marked 2-categories}. Taking the above reformulation of Theorem A as a model, this paper proposes and proves the following theorem 

\begin{theorem}[Theorem A\textsuperscript{$\dagger$}]\label{thm:AdagIntro}
	Let $F:\CC^\dagger \to \DD^\dagger$ be a functor of 2-categories with weak equivalences. Suppose that, for every object $d\in \DD$, 
	\begin{enumerate}
		\item There exists and object $c\in \CC$ and a marked morphism \begin{tikzcd}
			d \arrow[r,circled] & F(c)
		\end{tikzcd}
		\item Every marked morphism \begin{tikzcd}
			d \arrow[r,circled] & F(c)
		\end{tikzcd} 
		is initial in the $\infty$-categorical localization $L_\mathcal{W}(\Nerv_2(\CC_{d\downslash})^\dagger)$. 
	\end{enumerate}
	Then the induced functor $\func{F_{\mathcal{W}}: L_\mathcal{W}(\Nerv_2(\CC^\dagger))\to L_\mathcal{W}(\Nerv_2(\DD^\dagger))}$ is an equivalence of $\infty$-categories. 
\end{theorem} 

We will actually prove \autoref{thm:AdagIntro} as a corollary of another, seemingly more general result, \autoref{thm:thebigone}. It will turn out that these two theorems are, in fact, equivalent. In homotopy-theoretic situtations, \autoref{thm:AdagIntro} will tend to be more computationally tractable, and has clearer connections to established ideas in the literature. However, \autoref{thm:thebigone} has a compelling connection to notions of cofinality in 2-categories, leading us to the \emph{cofinality conjecture}, which will be discussed both at the end of the introduction, and in the final section of this paper.

The proof that our conditions are sufficient is quite straightforward, and offers an abstract way to construct a weak inverse to $F_\mathcal{W}$. The only price of this directness is that the proof relies on established $(\infty,2)$-categorical technology --- in particular the relative 2-nerve construction and locally Cartesian fibrations of simplicial sets. 

\subsection*{Applications}
\addcontentsline{toc}{subsection}{Applications}

Since \autoref{thm:AdagIntro} involves generalizations of the classical Theorem A in two ways --- introducing both 2-categories and a choice of marked morphisms to the picture --- it is unsurprising that there are a number of special cases which themselves constitute interesting generalizations of Quillen's Theorem A. The first of these forgets one of the generalizations --- that to strict 2-categories --- and focuses only on the marked morphisms. In this context, \autoref{thm:AdagIntro} immediately reduces to:

\begin{theorem}\label{thm:Adag1Cat}
	Let $\func{F:\mathcal{C}^\dagger\to\mathcal{D}^\dagger}$ be a functor between marked categories such that, for all $d\in \mathcal{D}^\dagger$, 
	\begin{itemize}
		\item there exists $c\in\mathcal{C}^\dagger$ and a marked morphism $\begin{tikzcd}
		d\arrow[r,circled] & F(c)
		\end{tikzcd}$
		\item every marked morphism $\begin{tikzcd}
		d\arrow[r,circled] & F(c)
		\end{tikzcd}$ represents an initial object in the localization $L_{\mathcal{W}}(\mathcal{C}_{d/})$.
	\end{itemize}
	Then $F$ induces an equivalence on $\infty$-categorical localizations $\func{F_{\mathcal{W}}:L_{\mathcal{W}}(\mathcal{C})\to[\simeq] L_{\mathcal{W}}(\mathcal{D})}$. 
\end{theorem}

Quillen's Theorem A then amounts to the special case in which all morphisms in $\mathcal{C}$ and $\mathcal{D}$ are marked. The localization $L_{\mathcal{W}}(\mathcal{C})$ can then be identified with the Kan-Quillen fibrant replacement, and so the conclusion of the theorem reduces to an equivalence of spaces.

There are also a number of results from the more recent literature that can be retrieved as special cases of \autoref{thm:AdagIntro}. In particular, if we instead remember the generalization to 2-categories, but neglect the generalization to marked morphisms (by considering every morphism to be marked), we obtain the following theorem of Bullejos and Cegarra from \cite{Bullejos_QuillenA} :

\begin{theorem}[Bullejos and Cegarra]\label{thm:IntroBC}
	Let $\func{F:\CC\to \DD}$ be a 2-functor. Suppose that, for every $d\in \DD$, there is a homotopy equivalence $|\Nerv_2(\CC_{d\downslash})|\simeq \ast$. Then $\func{|F|:|\Nerv_2(\CC)|\to |\Nerv_2(\DD)|}$ is a homotopy equivalence. 
\end{theorem}

Finally, there is also a criterion of Walde from \cite{Walde}, which checks when an $\infty$-categorical localization of a 1-category yields a 1-category as output. This is a special case of \autoref{thm:Adag1Cat} where the marking on the target 1-category consists of the equivalences. 

We will prove each of these corollaries, as well as some related criteria, in \autoref{sec:coraps}. However, it is quite informative to consider the 1-categorical case \autoref{thm:Adag1Cat} on its own, as it shares many of the salient features of the proof of the full 2-categorical case, but without many of the technical combinatorial complexities.

We will not present a separate proof of \autoref{thm:Adag1Cat} here, as doing so would have the effect of doubling many arguments unnecessarily. The interested reader can easily reconstruct the proof from the proof of \autoref{thm:AdagIntro}. Such a reconstruction is also rendered simpler by the disappearance of some technical, 2-categorical facets of the proof:
\begin{itemize}
	\item In the 1-categorical case, one need not take into account any convention for 2-morphisms. In particular, this obviates the need to work with lax overcategories or the Duskin nerve. 
	\item One can work with the relative 1-nerve $\chi$ of \cite{HTT} rather than the relative 2-nerve from \cite{ADS2Nerve}.
	\item Since the slices $\mathcal{C}_{d/}$ are 1-categories, we immediately get a Cartesian fibration, instead of having to first pass through passing through $L_{\mathcal{W}}$. 	
	\item The explicit section $s_{\mathcal{C}}$ constructed in \autoref{sec:sec} can, in the 1-categorical case, be written down immediately. As a result, the technology of Quillen adjunctions developed in \autoref{sec:QA} is unnecessary in this case. 
	\item The functor constructed from the section $s_{\mathcal{C}}$ is equal to the localization map of $\mathcal{C}$ on the nose, rendering the combinatorial argument relating the two in \autoref{sec:sec} moot in the 1-categorical case.
\end{itemize}

The technical difficulties which arise in the full 2-categorical proof are mostly due to the relative dearth of genuine $(\infty,2)$-categorical technology, as compared to the $(\infty,1)$-case. We expect that in the presence of a $(\infty,2)$-Grothendieck construction relating functors into $\Cat_{(\infty,2)}$ and marked-scaled Cartesian fibrations the arguments presented here would simplify greatly. It is worth noting that we expect the construction $\widetilde{\cchi}_{\DD}$ presented here to be the relative nerve construction corresponding to such a `genuine' $(\infty,2)$-Grothendieck construction.

\section*{Relation to cofinality}
\addcontentsline{toc}{subsection}{Relation to cofinality}
In its original form (\cite{QuillenK}) Theorem A is concerned with the homotopy-theoretic properties of functors between ordinary 1-categories. It would be only after the blossoming of homotopy coherent mathematics (i.e. model-category theory and  $\infty$-category theory) that the same statement could be more generally interpreted in terms of preservation of $\infty$-colimits. In this more modern framework one can recover that original result of Quillen by noting that the $\infty$-colimit of the constant point-valued functor
\[
\func*{\underline{*}: C \to \on{Top}}
\]
is the geometric realization of $C$. One would naturally expect that \autoref{thm:AdagIntro} follows the same pattern with a suitably categorified notion of colimits.

This paper can be considered as a first step in a longer program dedicated to a categorification of the cofinality criterion of Quillen's Theorem A. In \autoref{sec:cofinality} we explore the notion of \emph{marked colimits} in the setting of strict 2-categories, obtaining a decategorified cofinality criterion \autoref{thm:decatAdag}. This later result coupled with \autoref{thm:buckleycomputes} will then yield a strict version of the main result of this paper. 

A feature of great interest in both \autoref{thm:AdagIntro} and \autoref{thm:decatAdag} is that neither is merely a vertical categorification of Quillen's Theorem A. To pass one rung higher on the ladder of categorification, it is necessary to add structure in two directions: (1) categorical structure in the form of non-invertible 2-morphisms and (2) homotopical structure in the form of a chosen set of marked morphisms. The latter has a profound impact on the definition of marked colimits appearing in this paper. The notion of marked colimit is, as the name implies, highly sensitive to the marking on the diagram 2-category --- so much so, in fact, that even operations on the marking which do not change $\infty$-categorical localization (e.g. taking saturations) do not preserve marked colimits. 

This facet of the developing theory is not surprising, given that even in the $(\infty,1)$-context, a functor defining an equivalence of spaces is a far weaker condition than the same functor being cofinal. Every inclusion of an object into a category with a terminal object induces an equivalence of spaces, but such an inclusion is only cofinal when it selects a terminal object in the target category. As we develop the twinned notions of marked colimit and marked cofinality, we take care to comment on this sensitivity to marking, and to connect marked cofinality with the appropriate choice of hypotheses in the generalization of Theorem A. 

We conclude this work with a discussion of conjectures and open questions pointing towards a genuine theory of $(\infty,2)$-marked colimits. In particular, we propose the \emph{cofinality conjecture}, which posits a relation between a theory of $(\infty,2)$-marked colimits and the hypotheses of \autoref{thm:thebigone}. The results of this paper, in addition to their independent utility in computing $(\infty,1)$-categorical localizations of 1- and 2-categories, provide compelling evidence that a form of the cofinality conjecture should hold once the necessary technology has been developed.  

\section*{Acknowledgements}
The authors would like to thank Prof. Tobias Dyckerhoff for his guidance and support throughout this process. F.A.G. acknowledges the support of the VolkswagenStiftung through the Lichtenberg
Professorship Programme.

\newpage 
\section{Preliminaries}

In this section, we will go over the notations, definitions, and background necessary for our constructions and proofs. We focus for the most part on 2-categorical background, directing readers in need of higher-categorical and model-categorical preliminaries to \cite{HTT}, \cite{Cisinski}, \cite{Lurie_Goodwillie}, and \cite{ADS2Nerve}.

\subsection{2-categories and the 2-nerve}

\begin{notation} \glsadd{CCopop}
	By a \emph{2-category}, we will always mean a strict 2-category. By a \emph{2-functor}, we will mean a strict 2-functor unless specified otherwise. For a 2-category $\CC$, we will denote the 1-morphism dual by $\CC^{(\op,-)}$, the 2-morphism dual by $\CC^{(-,\op)}$, and the dual which reverses both 1- and 2-morphisms by $\CC^{(\op,\op)}$. We will denote the (1-)category of 2-categories and (strict) 2-functors by $2\!\Cat$.  
\end{notation}

\begin{definition}
	Let $\CC$ and $\DD$ be 2-categories. A \emph{normal lax 2-functor} (which we will sometimes refer to as simply a \emph{lax functor}) $\func{F:\CC\to \DD}$ consists of the data: 
	\begin{itemize}
		\item A map $\func{F:\on{Ob}(\CC)\to \on{Ob}(\DD)}$ on objects. 
		\item For each pair of objects $b,c\in \CC$, a functor 
		\[
		\func{F_{b,c}:\CC(b,c)\to \CC(F(b),F(c))}.
		\]
		\item For each triple of objects $a,b,c\in \CC$, a natural transformation 
		\[
		\begin{tikzcd}
		\CC(b,c)\times \CC(a,b)\arrow[r,"\circ"]\arrow[d,"F"'] & \CC(a,c)\arrow[d,"F"] \arrow[dl,Rightarrow ,"\sigma"]\\
		\DD(F(b),F(c))\times \DD(F(a),F(b))\arrow[r,"\circ"'] & \DD(F(a),F(c))
		\end{tikzcd}  
		\]
		called the \emph{compositor} of $F$. 
	\end{itemize}
	subject to the conditions 
	\begin{enumerate}
		\item $F_{c,c}(\id_c)=\id_{F(c)}$ for all $c\in \CC$. 
		\item $\sigma_{f,\id}=\id_{F(f)}$ and $\sigma_{\id,f}=\id_{F(f)}$. 
		\item the compositors satisfy the hexagon identity. 
	\end{enumerate}
	We denote by $\Lcat$ the (1-)category of 2-categories with normal lax functors as morphisms. 
\end{definition}

\begin{definition}
	Composing the cosimplicial object 
	\[
	\func*{\Delta^\bullet:\Delta\to \Cat; 
		[n]\mapsto [n]}
	\]
	with the inclusion $\func{\Cat\to \Lcat}$ yields a cosimplicial object in $\Lcat$, which we will also denote by $\Delta^\bullet$ in a slight abuse of notation. Using this cosimplicial object, we obtain a functor
	\[
	\func{\Nerv_2: \Lcat\to \Set_\Delta} 
	\]
	with $\Nerv_2(\CC)_n=\Lcat([n],\CC)$. We call this functor the \emph{(Duskin) 2-nerve}. \glsadd{N2}
\end{definition}

\begin{definition}\label{defn:OI}
	Let $I$ be a linearly ordered finite set. We define a $2$-category $\OO^{I}$ as
	follows
	\begin{itemize}
		\item the objects of $\OO^I$ are the elements of $I$,
		\item the category $\mathbb{O}^{I}(i,j)$ of morphisms between objects $i,j \in I$ is defined
		as the poset of finite sets $S \subseteq I$ such that $\min(S)=i$ and $\max(S)=j$
		ordered by inclusion,
		\item the composition functors are given, for $i,j,l\in I$, by
		\[
		\mathbb{O}^{I}(i,j) \times \mathbb{O}^{I}(j,l) \to \mathbb{O}^{I}(i,l), \quad (S,T) \mapsto S \cup T.
		\]
	\end{itemize}
	When $I=[n]$, we denote $\OO^I$ by $\OO^n$. Note that the $\OO^n$ form a cosimplicial object in $2\!\Cat$, which we denote by $\OO^\bullet$. \glsadd{OI}
\end{definition}

\begin{construction}
	We abuse notation and also denote by $\OO^\bullet$ the cosimplicial object 
	\[
	\begin{tikzcd}
	\Delta\arrow[r,"\OO^\bullet"] & 2\!\Cat \arrow[r] & \Lcat. 
	\end{tikzcd}
	\]
	For each $n\in \NN$, we can define a lax functor 
	\[
	\func{\xi_n:[n]\to \OO^n}
	\]
	defined as the identity on objects. On 1-morphisms we set $\xi_n(i\leq j)=\{i,j\}$. Since the $\Hom$-categories in $\OO^n$ are posets, this completely determines the compositors. 
\end{construction}

We now summarize some useful properties of the Duskin 2-Nerve:

\begin{proposition}
	Let $\CC$ be a 2-category.
	\begin{enumerate}
		\item The functor $\func{\Nerv_2:\Lcat\to \Set_\Delta}$ is fully faithful. 
		\item The lax functors $\xi_n:[n]\to \OO^n$ form a natural transformation of cosimplicial objects. 
		\item For every $n\in \NN$ the lax functor $\xi_n:[n]\to \OO^n$ induces a bijection 
		\[
		2\!\Cat(\OO^n,\CC)\cong \Lcat([n],\CC).
		\]
	\end{enumerate}
\end{proposition}

\begin{proof}
	See \cite{Bullejos_2Nerve} or \cite{kerodon} for the first statement. The second can be easily checked by hand, and the third is, e.g., \cite[2.3.6.7]{kerodon}.
\end{proof}

\begin{remark}
	The proposition above implies, in particular, that there are two ways of viewing a simplex $\func{\sigma:\Delta^n\to \Nerv_2(\CC)}$. We can either view it as a strict functor $\func{\OO^n\to\CC}$ or a normal lax functor $\func{[n]\to \CC}$. Note that, passing to the 2-nerve of the latter gives precisely the original inclusion $\func{\Delta^n\to \Nerv_2(\CC)}$. We will make extensive use of both conventions in our constructions. 
\end{remark}

\begin{remark}
	For any 2-category $\CC$, one can equip $\Nerv_2(\CC)$ with the additional structure of a scaling (see, e.g. \cite{ADS2Nerve} for details). The resulting functor $\func{2\!\Cat\to \Set_\Delta^{\on{sc}}}$ will be denoted by $\Nerv^{\on{sc}}$, and referred to as the \emph{scaled nerve}.
\end{remark}

\subsection{Lax slices and functoriality}

\begin{definition}\label{def:laxslice}\glsadd{lslice}
	Let $\CC$ be a 2-category, and $c\in \CC$. We define a 2-category $\CC_{c\downslash}$, the \emph{lax slice category}, as follows. The objects of $\CC_{c\downslash}$ are morphisms $\func{f:c\to d}$ in $\CC$. A morphism in $\CC_{c\downslash}$ from $\func{f:c\to d}$ to $\func{g:c\to b}$ consists of a morphism $\func{u:d\to c}$ in $\CC$ together with a 2-morphism $\func{\beta:g\nat u\circ f}$. A 2-morphism in $\CC_{c\downslash}$ from $(u,\beta)$ to $(v,\alpha)$ consists of a 2-morphism $\func{\nu:u\nat v}$ such that the induced diagram 
	\[
	\begin{tikzcd}[row sep =0.5em]
	& u\circ f\arrow[dd,Rightarrow] \\
	g\arrow[ur,Rightarrow,"\beta"]\arrow[dr,Rightarrow,"\alpha"'] &  \\
	& v\circ f
	\end{tikzcd}
	\]
	commutes. 
	
	We define the \emph{oplax slice category} $\CC_{c\upslash}$ to be the 2-category $((\CC^{(-,\op)})_{c\downslash})^{(-,\op)}$. This amounts to simply reversing the convention for the direction of the 2-morphisms filling the triangles in the morphisms. 
\end{definition}

\begin{remark}
	We here warn the reader that conventions related to the 2-morphism dual of a 2-category need to be carefully calibrated when dealing with the 2-nerve. The reason for this is that, since the 2-nerve encodes 2-morphisms as 2-simplices $\func{h\nat f\circ g}$, it enforces a choice of convention. It is not easy to define something like the 2-morphism dual of a scaled simplicial set (in contrast to the 1-categorical case, where the 1-morphism dual merely corresponds to the opposite simplicial set). 
	
	The convention forced by the 2-nerve is why, in the above discussion, we defined the compositors of lax functors to point in the direction we did. This convention is also the reason that our higher-categorical proofs will usually involve the slice categories $\CC_{c\downslash}$, rather than $\CC_{c\upslash}$. 
\end{remark}

\begin{proposition}
	We denote by $\Lcat_\ast$ the (strictly) pointed version of $\Lcat$. The assignment 
	\[
	\func*{\Lcat_\ast \to \Lcat; (\CC,c)\mapsto \CC_{c\downslash}}
	\]
	defines a functor. 
\end{proposition}

\begin{proof}
	We start with a functor $F:\CC\to \DD$ which maps $c$ to $F(c)$, and define a lax functor 
	\[
	\func*{F_{\downslash}:\CC_{c\downslash}\to \DD_{F(c)\downslash}}
	\]
	as follows. 
	\begin{itemize}
		\item On objects, we send $\func{f:c\to c_1}$ to $\func{F(f):F(c)\to F(c_1)}$. 
		\item On 1-morphisms we define 
		\[
		\begin{tikzcd}
		& c\arrow[dl,"f_1"']\arrow[dr,"f_2", ""{name=U,inner sep=1pt,below left}] & \\
		|[alias=X]|c_1\arrow[rr,"h"']& &c_2 \arrow[Rightarrow, from=U,
		to=X, shorten <=.3cm,shorten >=.6cm, "\;\;\;\;\;\mu"']
		\end{tikzcd}\mapsto \begin{tikzcd}
		& F(c)\arrow[dl,"F(f_1)"']\arrow[dr,"F(f_2)", ""{name=U,inner sep=1pt,below left}] & \\
		|[alias=X]|F(c_1)\arrow[rr,"h"']& &c_2 \arrow[Rightarrow, from=U,
		to=X, shorten <=.3cm,shorten >=.6cm, pos=0.1, "{\sigma(h,f_1)\circ F(\mu)}"',sloped]
		\end{tikzcd}
		\]
		where $\sigma(h,f_1)$ is the compositor of $F$. 
		\item on 2-morphisms, we simply send $\alpha\mapsto F(\alpha)$. 
	\end{itemize}
	We then note that the well-definedness of $\Nerv_2(F):\Nerv_2(\CC)\to \Nerv_2(\DD)$ on 3-simplices (see \cite{Bullejos_2Nerve}) implies that the compositor for $F$ defines 2-morphism in $\DD_{F(c)\downslash}$ 
	\[
	\func{\sigma(g_2,g_2):F_{\downslash}(g_2\circ g_1)\nat F_{\downslash}(g_2)\circ F_{\downslash}(g_1)} 
	\] 
	(where we abuse notation by denoting a morphism in $\CC_{c\downslash}$ by its projection to $\CC$). Since identities on 2-morphisms in $\DD_{F(c)\downslash}$ can be checked in $\DD$, it is immediate that this defines a normal lax functor. 
	
	The composability and unitality of the assignment $F\mapsto F_{\downslash}$ can then be checked immediately from the definitions.  
\end{proof}

\begin{remark}
	There are forgetful strict 2-functors $\CC_{c\downslash}\to \CC$ defined in the obvious way. When pieced together, these form a natural transformation to the forgetful functor $\Lcat_\ast\to\Lcat$. 
\end{remark}

\subsection{$\infty$-Localizations \& conventions for simplicial sets}

\begin{definition}\glsadd{Cdag}
	A \emph{marked 2-category} is a pair $\CC^\dagger=(\CC,W_{\CC})$ consisting of a 2-category $\CC$ together with a set $W_{\CC}$ of morphisms in $\CC$ containing all identities. A \emph{functor of marked 2-categories}  (\emph{marked 2-functor}) $\func{F:\CC^\dagger \to \DD^\dagger}$ is a 2-functor $\func{F:\CC\to \DD}$ such that $F(W_{\CC})\subset W_{\DD}$. We will denote the category of marked 2-categories with marked functors by $2\!\Cat^\dagger$.   
\end{definition}

\begin{definition}
	A \emph{category with weak equivalences} is a pair $\mathcal{C}^\dagger=(\mathcal{C},\mathcal{W}_\mathcal{C})$ consisting of a 1-category $\mathcal{C}$ and a wide subcategory $\mathcal{W}_{\mathcal{C}}\subset \mathcal{C}$. A \emph{homotopical functor} $\func{F:\mathcal{C}^\dagger\to \mathcal{D}^\dagger}$ is a functor $\func{F:\mathcal{C}\to \mathcal{D}}$ which sends $\mathcal{W}_\mathcal{C}$ into $\mathcal{W}_\mathcal{D}$. 
	
	A \emph{2-category with weak equivalences} $\DD^\dagger$ consists of a 2-category $\DD$ together with the structure of a category with weak equivalences on the underlying 1-category of $\DD$. A \emph{homotopical 2-functor} (or \emph{functor of marked 2-categories}) $\func{F:\CC^\dagger \to \DD^\dagger}$ is a strict 2-functor $\func{F:\CC\to \DD}$ such that the induced functor on underlying 1-categories is a marked functor. We will denote the category of 2-categories with weak equivalences and homotopical 2-functors by $2\!\Cat^{\on{we}}$. 
\end{definition}

\begin{remark}\label{rmk:widening}
	There is an obvious inclusion $\func{2\!\Cat^{\on{we}}\to 2\!\Cat^{\dagger}}$. This inclusion has a left adjoint $\func{Q:2\!\Cat^\dagger \to2\!\Cat^{\on{we}}}$, which we refer to as the \emph{widening functor}. For a marked 2-category $\CC^\dagger=(\CC,W_{\CC})$, the widening $Q(\CC^\dagger)$ has the same underlying 2-category. The subcategory of weak equivalences of $Q(\CC^\dagger)$ is the closure of $W_{\CC}$ under composition. 
\end{remark}

\begin{definition}\label{def:saturated}
	A marked 2-category $(\CC,\mathcal{W}_{\CC})$ is said to be \emph{saturated} if:
	\begin{itemize}
		\item The pair $(\CC,\mathcal{W}_{\CC})$ is a 2-category with weak equivalences.
		\item The category $\mathcal{W}_{\CC}$ contains all of the equivalences of $\CC$.
		\item Given $f \in \mathcal{W}_{\CC}$ and $g \in \CC$ together with an invertible 2-morphism $\func{f \nat[\simeq] g}$, then $g\in \mathcal{W}$.
	\end{itemize}
\end{definition}

\begin{definition}\label{def:markedslice}
	Let $\func{F: \CC^{\dagger}\to \DD^{\dagger}}$ be a functor of marked 2-categories. For every $d \in \DD^{\dagger}$ we define a marked 2-category $\CC_{d \downslash}^{\dagger}$, whose underlying 2-category is the lax slice category (\autoref{def:laxslice}), by declaring an edge to be marked if and only if the associated 2-morphism is invertible and the associated 1-morphism is marked in $\CC^{\dagger}$.  
\end{definition}

\begin{remark}
	Note that the 2-nerve $\func{\Nerv_2:2\!\Cat\to \Set_\Delta}$ extends to a functor 
	\[
	\func{\Nerv_2^\dagger:2\!\Cat^\dagger\to \Set_\Delta^+}
	\] 
	into marked simplicial sets. 
\end{remark}

\begin{definition}\glsadd{LW}
	Let $(X,W)\in \Set_\Delta^+$ be a marked simplicial set. A \emph{($\infty$-categorical) localization of $X$ by $W$} is an $\infty$-category $L_{\mathcal{W}}(X)$ together with a map $\func{\gamma_X:X\to L_{\mathcal{W}}(X)}$ of marked simplicial sets such that, for every $\infty$-category $\mathscr{C}$, the induced map  
	\[
	\func{\gamma_X^\ast:\Fun((X,W),\mathscr{C})\to[\simeq] \Fun(L_{\mathcal{W}}(X),\mathscr{C})}
	\]
	is an equivalence of $\infty$-categories. 
\end{definition}

\begin{remark}
	It is easy to see that fibrant replacement in the model structure on marked simplicial sets gives a localization map. We can therefore assume that $\func{L_{\mathcal{W}}:\Set_\Delta^+\to \Set_\Delta^+}$ is a functor, and that there is a canonical natural transformation $\func{\id_{\Set_\Delta^+}\nat L_{\mathcal{W}}}$ giving the localization morphism $\gamma_X$. 
\end{remark}

\begin{notation}
	Let $\CC^\dagger \in 2\!\Cat^\dagger$. We will denote by $L_{\mathcal{W}}(\DD^\dagger)$ the $\infty$-category $L_{\mathcal{W}}(\Nerv_2^\dagger(\DD^\dagger))$.
\end{notation}

\begin{proposition}
	Let $(X,W)\in \Set_\Delta^+$ be a marked simplicial set. The localization map 
	\[
	\func{\gamma_X:X\to L_{\mathcal{W}}(X)}
	\] 
	is both cofinal and coinitial. 
\end{proposition}

\begin{proof}
	This is \cite[Prop. 7.1.10]{Cisinski}. 
\end{proof}

\begin{remark}
	When we write \emph{cofinal}, we follow the convention of \cite{HTT}, in that cofinal functors are those $f:X\to Y$ such that precomposition with $f$ preserves $\infty$-colimits. We will call the dual notion (regarding preservation of $\infty$-limits) a \emph{coinitial functor}. 
\end{remark}

\section{Another model for the relative 2-nerve}\label{sec:QA}
In this section we  develop the necessary  technology for the proof of \autoref{thm:AdagIntro}. Recall that in \cite{ADS2Nerve} we constructed a Quillen equivalence

\[
\PPhi_{\CC}: (\msSet)_{/\Nsc(\CC)} \llra \Fun_{\msSet}(\CC^{(\op,\op)}, \msSet): \cchi_{\CC}
\]
for every 2-category $\CC$, where the left-hand side is equipped with the \emph{scaled cartesian} model structure and the right-hand side is equipped with the \emph{projective} model structure on $\msSet$-enriched functors. The main goal of this section is to  define a variant of $\cchi$ inducing a Quillen equivalence
\[
\widetilde{\PPhi}_{\CC}: (\msSet)_{/\Nsc(\CC)} \llra \Fun_{\msSet}(\CC^{(\op,\op)}, \msSet): \widetilde{\cchi}_{\CC}
\]
The functor $\widetilde{\cchi}$ will play a crucial role in the proof of \autoref{thm:AdagIntro} by allowing us to handle 2-categorical information in a more efficient way. In addition we will show that both constructions are related by means of a canonical comparison map $\func{\cchi \nat \widetilde{\cchi}}$. The main theorem of this section identifies a key property of this comparison: 

\begin{theorem}\label{thm:comparison}
	Let $\mathbb{C}$ be a 2-category. Then, for every $\msSet$-enriched functor
	\[
	\func{F: \CC^{(\op,\op)} \to \Cat_{\infty}}
	\]
	the comparison map  $\func{\eta_{\CC}: \cchi_{\CC}(F) \to \widetilde{\cchi}_{\CC}(F)}$ is a weak equivalence of scaled cartesian fibrations over $\Nsc(\CC)$.
\end{theorem}

We will review the basic definitions involved in the construction of the relative 2-nerve.

\begin{definition}
	We define the homotopy category functor 
	\[
	\func{\ho: 2\Cat \to \Cat }
	\]
	which sends a 2-category $\CC$ to the 1-category $\ho(\CC)$ having the sames objects as $\CC$ and with $\Hom$-sets given by $\ho(\CC)(c,c')=\pi_{0}\CC(c,c')$.
\end{definition}

\begin{definition}
	Let $I$ be a linearly ordered finite set such that $i=\min(I)$. We  denote by $D^{I}$ the homotopy category of $\OO^{I}_{i \downslash}$.
\end{definition}

The objects of $\OO^{I}_{i \downslash}$ can be identified with  subsets $S \subseteq I$ such that $\min(S)=i$. Recall that the mapping category $\OO^{I}_{i \downslash}(S,T)$ is a poset whose objects $\func{\mathcal{U}:S \to T}$ are given by subsets $\mathcal{U} \subseteq I$ such that
\[
\min(\mathcal{U})=\max(S), \enspace \max(\mathcal{U})=\max(T), \enspace T \subseteq S \cup \mathcal{U},
\] 
ordered by inclusion. Since $\mathcal{U} \subseteq [\max(S),\max(T)]$ it follows that $\OO^{I}_{i \downslash}(S,T)$ is contractible. In particular we obtain that the homotopy category $D^{I}$ is a poset. We have also proved the following lemma.

\begin{lemma}\label{lem:Joyal}
	The canonical map $\func{\Nerv_2(\OO^{I}_{i \downslash} )\to \Nerv(D^{I})}$ is  a weak equivalence in the Joyal model structure.
\end{lemma}

\begin{remark}\label{rem:lift}
	Recall from \cite{ADS2Nerve} that given an  inclusion of finite linearly ordered sets $J \subseteq I$ we obtain the following fully faithful pullback functors
	\[
	\func{ \rho_{J,I}: \OO^{I}(\min(I),\min(J))^{\op} \times D^{J} \to D^I;  (L,S) \mapsto L \cup S  }
	\]
	We observe that we can produce a lift of the previous functor to a commutative diagram
	\[
	\begin{tikzcd}[ampersand replacement=\&]
	\OO^{I}(\min(I),\min(J))^{\op} \times \OO^{J}_{j \downslash} \arrow[d] \arrow[r,"\widetilde{\rho}_{J,I}"] \& \OO^{I}_{i \downslash} \arrow[d] \\
	\OO^{I}(\min(I),\min(J))^{\op} \times D^{J} \arrow[r] \& D^I
	\end{tikzcd}
	\]
	One immediately checks that the functor $\widetilde{\rho}_{J,I}$ is injective on objects, 1-morphisms and 2-morphisms.
\end{remark}

We can now give the definition of the relative 2-nerve and its new variant.

\begin{definition}\label{def:rel2nerve}	\glsadd{cchi}
	Let $\CC$ be a $2$-category and let 
	\[
	F: \CC^{(\op,\op)} \lra \msSet
	\]
	be a $\msSet$-enriched functor. 
	We define a marked simplicial set $\cchi_{\CC}(F)$, called the {\em relative $2$-nerve of $F$}, as follows.
	An $n$-simplex of $\cchi_{\CC}(F)$ consists of 
	\begin{enumerate}
		\item an $n$-simplex $\sigma: \Delta^n \to \Nsc(\CC)$,
		\item for every nonempty subset $I \subset [n]$, a map of marked simplicial sets
		\[
		\theta_I: \Nerv(D^{I})^{\flat} \to  F(\sigma(\min(I))),
		\]
	\end{enumerate}
	such that, for every $J \subset I \subset [n]$, the diagram 
	\[
	\begin{tikzcd}
	\Nerv(\OO^{I}(\min(I),\min(J))^{\op})^{\flat} \times \Nerv(D^{J})^{\flat}
	\arrow{r}{\rho_{J,I}} \arrow[d,"\Nerv(\sigma) \times \theta_J"] & 
	\Nerv(D^{I})^{\flat} \arrow[d,"\theta_I"] \\
	\Nerv(\CC(\sigma(\min(I)),\sigma(\min(J)))^{\op})^{\flat} \times F(\sigma(\min(J)))
	\arrow{r}{F(-)} & 
	F(\sigma(\min(I)))
	\end{tikzcd}
	\]
	commutes. The marked edges of $\cchi_{\CC}(F)$ are defined as follows: An edge $e$ of $\cchi_{\CC}(F)$ consists of a
	morphism $f: x \to y$ in $\CC$, together with vertices $A_x \in F(x)$, $A_y \in F(y)$, and
	an edge $\widetilde{e}: A_x \to F(f)(A_y)$ in $F(x)$. We declare $e$ to be marked if
	$\widetilde{e}$ is marked. Finally, we consider $\cchi_{\CC}(F)$ as a simplicial set over
	$\Nsc(\CC)$ by means of the forgetful functor.
\end{definition}

Our strategy is to mirror \autoref{def:rel2nerve}, employing as our building blocks the 2-categories $\OO^{I}_{i \downslash}$. We take \autoref{rem:lift} as a guiding principle to lift the compatibility  conditions to this new version. 

\begin{definition}\label{def:chitilde}\glsadd{tilcchi}
	Let $\CC$ be a $2$-category and let 
	\[
	F: \CC^{(\op,\op)} \lra \msSet
	\]
	be a $\msSet$-enriched functor. 
	We define a marked simplicial set $\widetilde{\cchi}_{\CC}(F)$ as follows.
	An $n$-simplex of $\widetilde{\cchi}_{\CC}(F)$ consists of
	\begin{enumerate}
		\item an $n$-simplex $\sigma: \Delta^n \to \Nsc(\CC)$,
		\item for every nonempty subset $I \subset [n]$, a map of marked simplicial sets
		\[
		\theta_I: \NN\left(\OO^{I}_{i \downslash}\right)^{\flat} \to  F(\sigma(\min(I))),
		\]
	\end{enumerate}
	such that, for every $J \subset I \subset [n]$, the diagram 
	\[
	\begin{tikzcd}
	\Nerv(\OO^{I}(\min(I),\min(J))^{\op})^{\flat} \times \NN\left(\OO^{J}_{j \downslash}\right)^{\flat}
	\arrow{r}{\widetilde{\rho}_{J,I}} \arrow[d,"\Nerv(\sigma) \times \theta_J"] & 
	\NN\left(\OO^{I}_{i \downslash}\right)^{\flat} \arrow[d,"\theta_I"] \\
	\Nerv(\CC(\sigma(\min(I)),\sigma(\min(J)))^{\op})^{\flat} \times F(\sigma(\min(J)))
	\arrow{r}{F(-)} & 
	F(\sigma(\min(I)))
	\end{tikzcd}
	\]
	commutes. The marked edges are defined in a way totally analogous to those of $\cchi_{\CC}(F)$. Finally, we consider $\widetilde{\cchi}_{\CC}(F)$ as a simplicial set over $\Nsc(\CC)$ by means of the forgetful functor.
\end{definition}

\begin{remark}
	Observe that the canonical maps $\func{\OO^{I}_{i \downslash} \to D^{I}}$ induce a natural transformation of functors $\func{\widetilde{\eta}: \cchi \nat \widetilde{\cchi}}$. We denote the adjoint of $\widetilde{\eta}$ by $\func{\widetilde{\epsilon}: \widetilde{\PPhi} \nat \PPhi}$.
\end{remark}

\begin{remark}
	Let $\func{f: \CC \to \DD}$ be a 2-functor. Then, we have a diagram
	\[
	\begin{tikzcd}[ampersand replacement=\&]
	\Fun_{\msSet}(\DD^{(\op,\op)},\msSet) \arrow[d,"\widetilde{\cchi}_{\DD}"] \arrow[r] \& \Fun_{\msSet}(\CC^{(\op,\op)},\msSet) \arrow[d,"\widetilde{\cchi}_{\CC}"] \\
	\left(\msSet\right)_{\big/ \Nsc(\DD)} \arrow[r] \& \left(\msSet\right)_{\big/ \Nsc(\CC)}
	\end{tikzcd}
	\]
	which commutes up to natural isomorphism, where the top horizontal morphism is given by restriction and the bottom horizontal morphism is given by pullback.
\end{remark}

In order to prepare ourselves for the proof of \autoref{thm:comparison} we show that $\widetilde{\cchi}$ preserves trivial fibrations, paralleling the strategy followed in \cite{ADS2Nerve}.

Let $\CC$ be 2-category and consider a pointwise trivial fibration of $\msSet$-enriched functors $\func{ F \to G}$. Note that the canonical map $\func{\widetilde{\eta}_{\CC}: \cchi_{\CC} \nat \widetilde{\cchi}_{\CC}}$ is an isomorphism  on 0-simplices, 1-simplices and marked edges. Therefore, we obtain solutions to lifting problems of  the form

\begin{equation}\label{eq:boundary}
\begin{tikzcd}[ampersand replacement=\&]
(\partial \Delta^n)^{\flat} \arrow[d] \arrow[r] \& \widetilde{\cchi}_{\CC}(F) \arrow[d] \\
(\Delta^{n})^{\flat} \arrow[r] \arrow[ur, dotted] \& \widetilde{\cchi}_{\CC}(G)
\end{tikzcd}
\end{equation}
\begin{equation}\label{eq:markedlift}
\begin{tikzcd}[ampersand replacement=\&]
\partial (\Delta^1)^{\sharp} \arrow[d] \arrow[r] \& \widetilde{\cchi}_{\CC}(F) \arrow[d] \\
(\Delta^{1})^{\sharp} \arrow[r] \arrow[ur, dotted] \& \widetilde{\cchi}_{\CC}(G)
\end{tikzcd}
\end{equation}
for $n\leq 1$. We are left to show the case $\func{(\partial \Delta^n)^{\flat} \to (\Delta^n)^{\flat}}$ for $n\geq 2$ so we can systematically ignore the markings.

Let $n \geq 0$. We note that \autoref{rem:lift} implies that for any inclusion of finite linearly ordered sets $I \subseteq [n]$ we obtain a cofibration of simplicial sets
\[
\func{\Nerv\left(\OO^{n}(0,\min(I))^{\op}\right) \times \NN\left(\OO^{I}_{i \downslash}\right) \to \NN\left(\OO^{n}_{0 \downslash}\right)}
\]
and denote its image by $\mathfrak{A}(I)$. We also denote by $\mathfrak{O}^{I}_{i \downslash}$ the nerve of $\OO^{I}_{i \downslash}$. Let us define the following simplicial set
\[
\mathfrak{S}^{n}=\bigcup\limits_{\partial \Delta^{I} \subset \partial \Delta^n}\mathfrak{A}(I) \subset \mathfrak{O}^{n}_{0 \downslash}.
\]
A totally analogous argument to that in \cite[Proposition 3.1.1]{ADS2Nerve} shows that lifting problems of the form (1) are in bijection with lifting problems of the form

\begin{equation}
\begin{tikzcd}[ampersand replacement=\&]
\mathfrak{S}^{n} \arrow[d] \arrow[r] \&  F(\sigma(0)) \arrow[d] \\
\mathfrak{O}^{n}_{0 \downslash} \arrow[r] \arrow[ur, dotted] \& G(\sigma(0))
\end{tikzcd}
\end{equation}
which can be solved since left column map is a cofibration. We have now proved \autoref{prop:trivfib} below.

\begin{proposition}\label{prop:trivfib}
	The functor $\func{\widetilde{\cchi}_{\CC}: \Fun_{\msSet}\left(\CC^{(\op,\op)},\msSet\right) \to \left(\msSet\right)_{\big/ \Nsc(\CC)}}$ preserves trivial fibrations.
\end{proposition}

\begin{definition}
	Let $n \geq 0$, we define $\msSet$-enriched functor
	\[
	\func{\mathfrak{O}^{n}:(\OO^{n})^{(\op,\op)} \to \msSet; i \mapsto \left(\mathfrak{O}^{[i,n]}\right)^{\flat}_{i \downslash}.}
	\]
	We also have marked variant of the previous definition
	\[
	\func{\left(\mathfrak{O}^{1}\right)^{\sharp}:(\OO^{1})^{(\op,\op)} \to \msSet; i \mapsto \left(\mathfrak{O}^{[i,1]}\right)^{\sharp}_{i \downslash}.}
	\]
	Finally, we define another marked simplically enriched functor
	\[
	\func{\mathfrak{D}^{n}:(\OO^{n})^{(\op,\op)} \to \msSet; i \mapsto \left(\D^{[i,n]}\right)^{\flat}}
	\]
	where $\D^{I}$ denotes the nerve of the poset $D^{I}$.
\end{definition}

\begin{remark}
	Let us note that the abovementioned functors come equipped with a canonical natural transformations 
	\[
	\func{\theta_n:\mathfrak{O}^{n} \to \mathfrak{D}^{n}}, \enspace \func{\theta_1^{\sharp}:\left(\mathfrak{O}^{1}\right)^{\sharp} \to \left(\mathfrak{D}^{1}\right)^{\sharp}}
	\]
	which are a levelwise weak equivalence of marked simplicial sets as observed in \autoref{lem:Joyal}.
\end{remark}

\begin{lemma}\label{lem:keyquillenlemma}
	Consider $\func{(\Delta^n)^{\flat} \to \Delta^n}$ as an object of $\left( \msSet \right)_{\big/ \Nsc(\OO^n)}$. Then the following holds
	\begin{enumerate}
		\item $\widetilde{\PPhi}_{\OO^n}((\Delta^n)^{\flat})\isom \mathfrak{O}^{n}$, $\PPhi_{\OO^n}((\Delta^n)^{\flat})\isom \mathfrak{D}^{n}$.
		\item The canonical map $\func{\theta_n: \mathfrak{O}^{n} \to \mathfrak{D}^{n}}$ (resp. $\theta_1^{\sharp}$) can be identified with $\widetilde{\epsilon}_{\OO^n}((\Delta^n)^{\flat})$ (resp. $\widetilde{\epsilon}_{\OO^1}((\Delta^1)^{\sharp})$) under the isomorphisms above.
	\end{enumerate}
\end{lemma}
\begin{proof}
	Immediate from unraveling the definitions.
\end{proof}

\begin{theorem}\label{thm:Qequivalence}
	Let $\CC$ be a 2-category. Then, then functor $\widetilde{\cchi}_{\CC}$ extends to a Quillen equivalence
	\[
	\widetilde{\PPhi}_{\CC}: (\msSet)_{/\Nsc(\CC)} \llra \Fun_{\msSet}(\CC^{(\op,\op)}, \msSet): \widetilde{\cchi}_{\CC}.
	\]
\end{theorem}
\begin{proof}
	We will show that for any object $\func{X \to \Nsc(\CC)}$ the map $\widetilde{\epsilon}_{\CC}(X)$ is a levelwise weak equivalence. Since $\widetilde{\cchi}_{\CC}$ preserves trivial fibrations by \autoref{prop:trivfib} this will in turn imply that $\widetilde{\cchi}_{\CC}$ preserves fibrations as well. In addition, we will have constructed an equivalence of left derived functors $\func{\mathbb{L}\widetilde{\PPhi}_{\CC} \nat \mathbb{L}\PPhi_{\CC}}$ yielding the result.
	
	It is not hard to show that the natural transformation $\widetilde{\eta}$ is compatible with base change. Therefore, invoking \autoref{lem:eq} below we reduce the problem to checking that $\widetilde{\epsilon}_{\OO^n}((\Delta^n)^{\flat})$ is a levelwise weak equivalence for $n\geq 0$ as well as $\widetilde{\epsilon}_{\OO^1}((\Delta^1)^{\sharp})$. This follows immediately from \autoref{lem:keyquillenlemma}.
\end{proof}

The main result of the section now follows as a corollary of the previous theorem.

\begin{corollary}\label{cor:fibres}
	Let $\mathbb{C}$ be a 2-category. Then, for every $\msSet$-enriched functor
	\[
	\func{F: \CC^{(\op,\op)} \to \Cat_{\infty}}
	\]
	the comparison map  $\func{\eta_{\CC}: \cchi_{\CC}(F) \to \widetilde{\cchi}_{\CC}(F)}$ is a weak equivalence of scaled cartesian fibrations over $\Nsc(\CC)$.
\end{corollary}
\begin{proof}
	It follows from \autoref{thm:Qequivalence} that $\widetilde{\cchi}_{\CC}$ preserves fibrants objects. Let $\func{F:\CC^{(\op,\op)} \to \Cat_{\infty}}$ be a marked simplicially enriched functor. It will suffice by \cite[Lemma 3.2.25]{Lurie_Goodwillie} to show that the map,
	\[
	\func{\cchi_{\CC}(F) \to \widetilde{\cchi}_{\CC}(F)}
	\]
	is an equivalences upon passage to fibers. This allows us to reduce to the case where $\CC=*$ is the terminal category. For the rest of the proof $\EuScript{B}$ will denote the image of $F$ at the unique object of $*$.
	
	Recall that when $\CC=C$ is a 1-category we constructed in \cite{ADS2Nerve} a comparison map $\func{\chi_{C} \nat \cchi_{C}}$ with the relative nerve. Passing to adjoints we obtain the following natural transformations,
	\[
	\func{\widetilde{\PPhi}_{C} \nat \PPhi_{C} \nat \phi_C}
	\] 
	which are levelwise weak equivalences. Specializing to the case of $C=*$ and passing again to right adjoints we obtain the following morphisms
	\[
	\func{\EuScript{B} \to \cchi_{*}(\EuScript{B}) \to \widetilde{\cchi}_{*}(\EuScript{B})}.
	\]
	We proved in \cite[Theorem 4.1.1]{ADS2Nerve} that the first map is a weak equivalence. To check that the composite map is a weak equivalence we can pass to adjoints. Therefore, the result follows from 2-out-of-3.
\end{proof}

\begin{lemma}\label{lem:eq}
	Let $\overline{K}=(K,K^{\operatorname{t}})$ be a scaled simplicial set. Suppose we are given two left adjoint functors, 
	\[
	L_1,L_2: (\msSet)_{/\overline{K}} \to C,
	\]
	where $C$ is a left proper combinatorial model category and $L_2$ is a left Quillen functor. Suppose further that $L_1$ preserves cofibrations. Given a natural transformation
	$\eta: L_1 \Rightarrow L_2$ which is a weak equivalence on objects of the form
	\[
	(\Delta^n)^{\flat} \to K \quad \text{$n \geq 0$,}
	\]
	and	
	\[
	(\Delta^n)^{\sharp} \to K, \quad \text{$n=1$.}
	\]
	Then $\eta$ is a levelwise weak equivalence.
\end{lemma}
\begin{proof}
	This is a special case of \cite[Lemma 4.3.3]{ADS2Nerve}.
\end{proof}

\section{Fibrations and sections}\label{sec:sec}

The proof of the main theorem of this paper will depend heavily on the properties of fibrations of simplicial sets. In the 1-categorical case, these would be Cartesian fibrations, but in full 2-categorical generality, our Grothendieck construction produces \emph{scaled Cartesian fibrations}. 

\begin{definition}
	Let $p:X\to Y$ be a map of simplicial sets. We call $p$ a \emph{locally Cartesian fibration} if $p$ is an inner fibration and, for every edge $\sigma:\Delta^1\to S$, the pullback $\func{p\times_Y \sigma: X\times_Y \Delta^1\to \Delta^1}$ is a Cartesian fibration. 
	
	We call an edge $f:\Delta^1\to X$ \emph{locally Cartesian} if it is a Cartesian edge of $p\times_Y (p\circ f)$.
\end{definition} 

\begin{remark}
	Dualizing the discussion in Example 3.2.9 of \cite{Lurie_Goodwillie}, we note that every scaled Cartesian fibration is in particular locally Cartesian. Also note that given a locally Cartesian fibration $p:X\to Y$, a morphism $f:a\to b$ in $Y$, and an object $\tilde{y}\in X_y$, there is a locally Cartesian morphism $\tilde{f}:\tilde{x}\to \tilde{y}$ lifting $f$. 
\end{remark}

We can now formulate and prove the property of locally Cartesian fibrations which will form the backbone of our proof of the main theorem. 

\begin{proposition}\label{prop:locCartTrivKanFib}
	Let $p:X\to S$ be a locally Cartesian fibration of simplicial sets. Assume that for each vertex $s\in S$, the $\infty$-category $X_s$ has an initial object. Denote by $X^\prime\subset X$ the full simplicial subset of $X$ spanned by those $x$ which are initial objects in $X_{p(x)}$. Then 
	\[
	p|_{\scr{X}^\prime}:\scr{X}^\prime\to S
	\]
	is a trivial Kan fibration of simplicial sets. Moreover, a section $q$ of $p:X\to S$ is initial in the $\infty$-category $\on{Map}_S(S,X)$ if and only if $q$ factors through $X^\prime$.  
\end{proposition}

Before we continue with the proof, note that this is identical to \cite[Prop. 2.4.4.9]{HTT} in every way except that we only require $p:X\to Y$ to be locally Cartesian. 

\begin{proof}
	The proof is effectively the same as that of \cite[Prop. 2.4.4.9]{HTT}. We comment on the points which differ. Tracing back through a sequence of lemmata\footnote{The dependency graph is 2.4.4.2$\to$2.4.4.7$\to$2.4.4.8$\to$2.4.4.9} we find that the only point where the fact that $p$ is a Cartesian fibration is used is to find a Cartesian lift of a morphism ending at a given object, and then apply Lemma 2.4.4.2. However, Lemma 2.4.4.2 only requires a \emph{locally} Cartesian lift, and so the proof runs through.  
\end{proof}

\subsection{Constructing the section}

We are now in the setting that we need, we can begin to perfom the key constructions needed in our proof . For the rest of this section, we fix a marked 2-category $\CC^{\dagger}$.

Consider the $\msSet$-enriched functor \glsadd{CCslash}
\[
\func*{\mathfrak{C}_{\CC \downslash}: \CC^{(\op,\op)} \to \msSet; c \mapsto \NN \left(\CC^{\dagger}_{c \downslash}\right) }
\]
and denote its image under  $\widetilde{\cchi}_{\CC}$ by $\func{p:\widetilde{\cchi}\left(\mathfrak{C}_{\CC \downslash}\right)\to \Nsc(\CC)}$. We will  define a  canonical section to  the map $p$.  

Let $\sigma:\OO^n\to \CC$ represent a simplex in $\Nsc(\CC)$. Given $I \subseteq [n]$ denote as usual $i=\min(I)$ and consider the following commutative diagram
\[
\begin{tikzcd}[ampersand replacement=\&]
\left(\OO^{I}_{i\downslash}\right)^{\flat} \arrow[rr,"\tau^{I}_{\sigma}"] \arrow[d] \&  \& \CC^{\dagger}_{\sigma(i)\downslash} \arrow[d] \\
\left(\OO^{I}\right)^{\flat} \arrow[r] \& \left(\OO^{n}\right)^{\flat} \arrow[r,"\sigma"] \&  \CC^{\dagger}
\end{tikzcd}
\]

The properties of the previous construction can be sumarized in the following proposition.
\begin{proposition}
	The assignment $\func{(\sigma: \OO^n \to \CC) \to (\sigma,\set{\tau_{\sigma}^{I}}_{I \subseteq [n]})}$ defines a map of simplicial sets
	\[
	\func{s_{\CC}:\Nsc(\CC) \to \widetilde{\cchi}_{\CC}(\mathfrak{C}_{\CC \downslash})}
	\]
	such that $p \circ s_{\CC}=\operatorname{id}$. The map $s_{\CC}$ sends an object $c\in \CC$ to the pair $(c,\on{id}_c)$.
\end{proposition}
\begin{proof}
	This is a special case of \autoref{prop:homotopyofsections} below. 
\end{proof}






We then compose the $\infty$-categorical localization functor $L_{\mathcal{W}}$ with the functor $\mathfrak{C}_{\CC \downslash}$ to get a projectively fibrant functor $\hat{\mathfrak{C}}_{\CC \downslash}=L_{\mathcal{W}}\circ \mathfrak{C}_{\CC \downslash}$. Denote by  $\func{\hat{p}:  \hat{\mathfrak{C}}_{\CC \downslash} \to \Nsc(\CC)}$ its image under $\widetilde{\cchi}_{\CC}$ and observe  that $\hat{p}$ is a scaled cartesian fibration. Moreover, we obtain a natural transformation
\[
\func{\theta: \mathfrak{C}_{\CC \downslash} \to \hat{\mathfrak{C}}_{\CC \downslash}}
\] 
of $\msSet$-enriched functors. We define our desired section to be $\hat{s}_{\CC}=\widetilde{\cchi}_{\CC}(\theta) \circ s_{\CC}$.

\begin{corollary}\label{cor:inital}
	There is a map of simplicial sets $\func{\hat{s}_{\CC}:\Nsc(\CC)\to \widetilde{\cchi}\left(\hat{\mathfrak{C}}_{\CC \downslash}\right)}$ such that $\hat{p} \circ \hat{s}_{\CC}=\on{id}$. The section sends an object $c\in \CC$ to the pair $(c,\on{id}_c)$.   
\end{corollary}

\subsection{The section as a localization map}

We now want to associate a composite of the section $\hat{s}_{\CC}:\Nsc(\CC)\to \widetilde{\cchi}\left(\hat{\mathfrak{C}}_{\CC \downslash}\right)$ with the localization map $\Nsc(\CC)\to L_W(\Nsc(\CC))$. Observe that we have the following commutative diagram of marked simplicial sets over $\Nsc(\CC)$

\[
\begin{tikzcd}[ampersand replacement=\&]
\widetilde{\cchi}_{\CC}\left( \mathfrak{C}_{\CC \downslash} \right) \arrow[r] \arrow[d] \& \NN\left(\CC\right) \times \widetilde{\cchi}_{*}\left(  \NN\left(\CC\right) \right) \arrow[d] \\
\widetilde{\cchi}_{\CC}\left( \hat{\mathfrak{C}}_{\CC \downslash} \right) \arrow[r] \&  \NN\left(\CC\right) \times \widetilde{\cchi}_{*}\left(L_W(\CC^{\dagger})\right)
\end{tikzcd}
\]
where we are implicity using the fact that for every constant $\msSet$-enriched functor $\widetilde{\cchi}_{\CC}(\underline{X})\isom \CC \times \widetilde{\cchi}_{*}(X)$. Our goal is to show that the composite map
\[
\begin{tikzcd}[ampersand replacement=\&]
\NN(\CC^{\dagger}) \arrow[r,"\hat{s}_{\CC}"] \& \widetilde{\cchi}_{\CC}\left( \hat{\mathfrak{C}}_{\CC \downslash} \right) \arrow[r,"\hat{\rho}_{\CC}"] \&  \widetilde{\cchi}_{*}\left(L_W(\CC^{\dagger})\right) 
\end{tikzcd}
\]
is a weak equivalence of marked simplicial sets. Let $\CC=*$, in \autoref{cor:fibres} we saw that the map of marked simplicial sets $\func{\NN(\CC^{\dagger}) \to L_W\left(\NN(\CC^{\dagger})\right)}$ induces a commutative diagram in $\msSet$

\[
\begin{tikzcd}[ampersand replacement=\&]
\NN(\CC^{\dagger}) \arrow[r] \arrow[d,"\simeq"] \& \widetilde{\cchi}_{*}\left( \NN(\CC^{\dagger})\right) \arrow[d] \\
L_W\left( \CC^{\dagger} \right) \arrow[r,"\simeq"] \& \widetilde{\cchi}_{*}\left(L_W\left( \CC^{\dagger} \right) \right)
\end{tikzcd}
\]
where the bottom row is a weak equivalence. To show that $\rho_{\CC}\circ s_{\CC}$ is a weak equivalence it will suffice to show that the following diagram commutes up to natural equivalence

\[
\begin{tikzcd}
& \widetilde{\cchi}_\CC\left(\mathfrak{C}_{\CC\downslash}\right)\arrow[d]\arrow[dd,out=-30, in=30, "\rho_{\CC}"]\\
N_2( \CC)\arrow[ur,"s_{\CC}"]\arrow[r]\arrow[d,"\simeq"] & \widetilde{\cchi}_{\ast}\left(\NN( \CC)\right)\arrow[d]\\
L_W \left( N_2( \CC)\right)\arrow[r,"\simeq"'] & \widetilde{\cchi}_\ast\left(L_W(\CC^{\dagger}) \right) 		
\end{tikzcd}
\]
The rest of this section is consequently devoted to produce a natural equivalence exhibiting commutativitivy of the upper triangle in the diagram above.

\begin{construction}\label{cons:partialcollapse}
	A simplex $i:\Delta^{n}\to \Delta^1$ is specified by a sequence 
	\[
	\overset{0}{0}\;0\;\cdots \;0\;\overset{i}{1}\;\cdots \; \overset{n}{1}
	\] 
	or simply by the number $0\leq i\leq n+1$. Given such a simplex, we define a 2-category $\on{P}^i\OO^n$, the \emph{$i$\textsuperscript{th} partial collapse of $\OO^n$} as follows.
	\begin{itemize}
		\item The objects of $\on{P}^i\OO^n$ are the objects of $\OO^n$ 
		\item The hom-categories of $\on{P}^i\OO^n$ are given by 
		\[
		\on{P}^i\OO^n(\ell,j)=\begin{cases}
		\emptyset & \ell> j\\
		\OO^n(\ell,j) & \ell \leq j < i\\
		\ast & \ell \leq i\leq j 
		\end{cases}
		\]
	\end{itemize} 
	The composition functors are those of $\OO^n$ where applicable, and the unique functor to $\ast$ everywhere else. 
	
	There is a canonical (strict) projection 2-functor $p_{i,n}:\func{\OO^n\to \on{P}^i\OO^n}$, and a canonical normal lax functor 
	\[
	\func*{\ell_{i,n}:\on{P}^i\OO^n\to \OO^n}
	\]
	which acts as the identity on objects, the identity on $\OO^n$ hom-categories, and, for $j\geq i$, sends 
	\[
	\func{({\ast:\ell\to j})\mapsto ({\{\ell,j\}:\ell\to j}).}
	\]
	Note that this functor fits into a (strictly) commutative diagram 
	\[
	\begin{tikzcd}
	\OO^{[0,i)}\arrow[dr,hookrightarrow]\arrow[d,hookrightarrow] & \\
	\on{P}^i\OO^n\arrow[r, "\ell_{i,n}"'] & \OO^n \\
	\Delta^{[i,n]}\arrow[u,hookrightarrow]\arrow[r, "\text{lax}"'] & {\OO^{[i,n]}}\arrow[u,hookrightarrow]
	\end{tikzcd}
	\]
\end{construction}

\begin{remark}
	Note that in the two extremal cases, we have $\on{P}^0\OO^n=\Delta^n$ and $\on{P}^{n+1}\OO^n=\OO^n$. We additionally have $p_{n+1,n}=\ell_{n+1,n}=\id_{\OO^n}$. 
\end{remark}

\begin{proposition}\label{prop:collapsedelta1}
	The assignment $\func{(i:\Delta^n \to \Delta^1) \to \on{P}^{i}\OO^n}$ described in \autoref{cons:partialcollapse} extends to a functor
	\[
	\func{\on{P}\OO^{\bullet}: \Delta_{/[1]} \to \Lcat.}
	\]
\end{proposition}
\begin{proof}
	Let $i,j  \in \Delta_{/[1]}$ and consider a morphism $\func{i \to[f] j}$. Let $[n]$ (resp. $[m]$) denote the source of the morphism $i$ (resp. $j$). Observe that we have the following commutative diagram
	\[
	\begin{tikzcd}[ampersand replacement=\&]
	\OO^{n} \arrow[r,"f"] \arrow[d,"p_{i,n}"] \& \OO^{m} \arrow[d,"p_{j,m}"] \\
	\on{P}^{i}\OO^n \arrow[r,dotted,"\overline{f}"] \& \on{P}^{j}\OO^{m}
	\end{tikzcd}
	\]
	where the upper row map is the usual action on morphims of the cosimplicial object $\OO^{\bullet}$. We note that the dotted arrow exists and it is unique due to fact $p_{j,m} \circ f$ is compatible with the collapses defining $\on{P}^{i}\OO^n$. Therefore we just set $\on{P}\OO^{\bullet}(f)=\overline{f}$. We omit the rest of the details.
\end{proof}

\begin{lemma}\label{lem:collapsefunctoriality}
	Let $\OO^{\bullet}$ denote the cosimplicial object of (REF) and define
	\[
	\begin{tikzcd}[ampersand replacement=\&]
	\OO^{\bullet}_{/[1]}: \Delta_{/[1]} \arrow[r] \& \Delta \arrow[r,"\OO^{\bullet}"] \& \Lcat.
	\end{tikzcd}
	\]
	Then, the various maps $p_{i,n}$ (resp. $\ell_{i,n}$) assemble into natural transformations 
	\[
	\func{\OO^{\bullet}_{/[1]} \nat[p] \on{P}\OO^{\bullet} \nat[l] \OO^{\bullet}_{/[1]}}
	\]
\end{lemma}

\begin{proof}
	Left as an exercise to the reader.
\end{proof}

\begin{definition}
	Let $\func{i:[n] \to[] [1]}$ be a monotone map and consider the 2-category $\on{P}^{i}\OO^{n}$. Given $I \subset [n]$ we define a 2-category $(\on{P}^{i} \OO^n )^{I}$ as follows
	\begin{itemize}
		\item The objects of $(\on{P}^i\OO^n)^{I}$ are the objects of $\OO^I$ 
		\item The hom-categories of $(\on{P}^i\OO^n)^{I}$ are given by 
		\[
		(\on{P}^i\OO^n )^{I}(\ell,j)=\begin{cases}
		\emptyset & \ell> j\\
		\OO^I(\ell,j) & \ell \leq j < i\\
		\ast & \ell \leq i\leq j 
		\end{cases}
		\]
	\end{itemize}
\end{definition}

\begin{remark}
	We observe that as a consequence of \ref{lem:collapsefunctoriality}  we have the following commutative diagram
	\[
	\begin{tikzcd}
	\OO^I\arrow[r]\arrow[d] & (\on{P}^i\OO^n)^I\arrow[d]\arrow[r] & \OO^I\arrow[d] \\
	\OO^n\arrow[r] & \on{P}^i\OO^n\arrow[r] & \OO^n.
	\end{tikzcd}
	\]
\end{remark}

\begin{construction}\label{const:homotopyofsections}
	Let $ \CC$ be a 2-category. We define a map of simplicial sets 
	\[
	H:\Nerv_2([1]\times  \CC)\to \widetilde{\cchi}_\ast(\NN( \CC))
	\]
	as follows. For a simplex in $\NN([1]\times  \CC)$ presented by a strict functor $\OO^n\to [1]\times  \CC$, we consider the component maps $i:\OO^n\to [1]$ and $\sigma:\OO^n\to  \CC$. We then send the simplex $(i,\sigma)$ to the data $\{\func{\tau^I:\OO^I_{\min(I)\downslash}\to  \CC}\}_{I\subset [n]}$ defined to be (the 2-nerve of) any of the composites in the diagram
	\[
	\begin{tikzcd}
	\OO^I_{\min(I)\downslash}\arrow[r]\arrow[d ] & (\on{P}^i\OO^n)^I_{\min(I)\downslash}\arrow[d]\arrow[r] & \OO^I_{\min(I)\downslash}\arrow[d] \\
	\OO^I\arrow[r ]\arrow[d] & (\on{P}^i\OO^n)^I\arrow[d]\arrow[r ] & \OO^I\arrow[d ] \\
	\OO^n\arrow[r] & \on{P}^i\OO^n\arrow[r] & \OO^n\arrow[r ,"\sigma"] &  \CC
	\end{tikzcd}
	\]
	going from $\OO^I_{\min(I)\downslash}$ in the upper left to $\CC$. 
\end{construction}

\begin{proposition}\label{prop:homotopyofsections}
	The map $H$ defined in Construction \ref{const:homotopyofsections} is a well-defined map of simplicial sets. 
\end{proposition}

\begin{proof}
	We first check that the data given do, indeed, form a simplex. in $\widetilde{\cchi}_\ast (\Nerv_2(\CC))$. This amounts to taking $\emptyset\neq J\subset I\subset [n]$ and showing that the diagram
	\[
	\begin{tikzcd}
	\OO^I(i,j)^\op\times \OO^J_{\min(J)\downslash}\arrow[d,"\tau^J"']\arrow[r] & \OO^I_{\min(I)\downslash}\arrow[d,"\tau^I"]\\
	\ast \times \CC\arrow[r,"\id"'] & \CC
	\end{tikzcd}
	\]
	commutes. We can check this on 0-, 1-, and 2-morphisms. On 0-morphisms, it is easy to see that both composites map a pair $(S,U)$ to $\sigma(\max(U)):=\sigma(\max(S\cup U))$. Likewise, on 1-morphisms, a pair $(T\subset S, W\cup U\supset V)$ is sent in both cases to $\sigma(\ell_{i,n}(p_{i,n}(W)))$. Finally, on 2-morphisms, we note that both composites will factor through $\on{P}^i\OO^n$. However, the hom-categories in $\on{P}^i\OO^n$ are either posets or points. Therefore, the values on 1-morphisms completely determine those on 2-morphisms. 
	
	The functoriality in $[n]$ follows immediately from \autoref{lem:collapsefunctoriality}. 
\end{proof}

\begin{remark}
	If $\sigma:\OO^1\to  \CC$ is constant on an object and $i:\OO^1\to [1]$  is the unique non-degenerate simplex, all of the data of the corresponding 1-simplex of $\widetilde{\cchi}_\ast(N_2( \CC))$ will factor through a constant map $\ast\to  \CC$. That is $H(i,\sigma)$ will be degenerate. Consequently, $H$ defines a natural equivalence. 
\end{remark}

Note that $H$ is thus a map of simplicial sets $\Delta^1\times \Nerv_2( \CC)\to \widetilde{\cchi}_\ast(\Nerv_2( \CC))$. On $0\times \Nerv_2( \CC)$, it sends a simplex $\sigma:\OO^n\to  \CC$ to the composite data 
\[
\{\tau^I: \OO^I_{\min(I)\downslash}\to \OO^n\overset{\sigma}{\to} \CC\}
\]
i.e. $H(0,-)$ is the composite 
\[
\begin{tikzcd}[ampersand replacement=\&]
\NN(\CC) \arrow[r,"s_{\CC}"] \& \widetilde{\cchi}_{\CC}(\mathfrak{C}_{\CC \downslash}) \arrow[r] \& \NN(C) \times \widetilde{\cchi}_{\ast}(\NN(\CC)) \arrow[r] \& \widetilde{\cchi}_{\ast}(\NN(\CC))
\end{tikzcd}
\]
Moreover, $H(1,-)$ sends a simplex $\sigma:\OO^n\to  \CC$ to the composite data 
\[
\{\func{\NN\left(\OO^I_{\min(I)\downslash}\right)\to \NN(\OO^n)\to\Delta^n \to \NN( \CC) } \}_{I \subset [n]}
\]
i.e. $H(1,-)$ is the composite 
\[
\begin{tikzcd}
N_2( \CC)\arrow[r,equals] &\chi_\ast(N_2( \CC)) \arrow[r] &\widetilde{\cchi}_\ast(N_2( \CC))
\end{tikzcd}
\]
and thus fits into a commutative square 
\[
\begin{tikzcd}
N_2( \CC)\arrow[r]\arrow[d,"\simeq"] & \widetilde{\cchi}_\ast(N_2( \CC))\arrow[d]\\
L_W ( \CC^{\dagger})\arrow[r,"\simeq"] & \widetilde{\cchi}_\ast(L_W ( \CC^{\dagger})) 		
\end{tikzcd}
\]
what we have thus shown is:

\begin{corollary}\label{cor:sectiongivesID}
	Let $ \CC$ be a 2-category. Then the diagram 
	\[
	\begin{tikzcd}
	& \widetilde{\cchi}_\CC(\mathfrak{C}_{\CC\downslash})\arrow[d]\arrow[dd,out=-30, in=30, "\rho_{\CC}"]\\
	\NN( \CC)\arrow[ur,"s_{\CC}"]\arrow[r]\arrow[d,"\simeq"] & \widetilde{\cchi}_ \CC(\NN( \CC))\arrow[d]\\
	L_W( \CC^{\dagger})\arrow[r,"\simeq"'] & \widetilde{\cchi}_\ast(L_W ( \CC^{\dagger})) 		
	\end{tikzcd}
	\]
	commutes up to natural equivalence.
\end{corollary}

\section{The main theorem}

The goal of this section will be to prove the main theorem of this paper: 

\begin{theorem}[Theorem A\textsuperscript{$\dagger$}]\label{thm:thebigone}
	Let $F:\CC^\dagger \to \DD^\dagger$ be a functor of marked 2-categories. Suppose that,
	\begin{enumerate}
		\item\label{asump:init} For every object $d\in \DD$, there exists a morphism $\func{g_d: d\to F(c)}$ which is initial in both $L_{\mathcal{W}}(\CC_{d\downslash}^\dagger)$ and $L_{\mathcal{W}}(\DD^\dagger_{d\downslash})$.
		\item\label{assump:markedinit} Every marked morphism \begin{tikzcd}
			d \arrow[r,circled] & F(c)
		\end{tikzcd}
		is initial in $L_{\mathcal{W}}(\CC_{d\downslash}^\dagger)$. 
		\item\label{asump:presundermarked}  For any marked morphism \begin{tikzcd}f:b\arrow[r,circled] & d\end{tikzcd} in $\DD$, the induced functors $\func{f^\ast: L_{\mathcal{W}}(\CC^\dagger_{d\downslash}) \to L_{\mathcal{W}}(\CC^\dagger_{b\downslash})}$ preserve initial objects.
	\end{enumerate}
	Then the induced functor $\func{F_{\mathcal{W}}: L_\mathcal{W}(\Nerv_2(\CC^\dagger))\to L_\mathcal{W}(\Nerv_2(\DD^\dagger))}$ is an equivalence of $\infty$-categories. 
\end{theorem} 

\begin{remark}
	Assumptions \ref{assump:markedinit} and \ref{asump:presundermarked} may seem a bit asymmetrical compared with assumption \ref{asump:init}, in that they do not involve $L_{\mathcal{W}}(\DD^\dagger_{d\downslash})$ However, as the next lemma shows, marked morphisms are \emph{always} initial in $L_{\mathcal{W}}(\DD^\dagger_{d\downslash})$. Since the identity is always initial in $L_{\mathcal{W}}(\DD^\dagger_{d\downslash})$, this immediately shows that precomposition with marked morphisms preserves one initial object, and thus preserves all initial objects.
\end{remark}

\begin{lemma}\label{lem:IDis2cofinal}
	Let $\DD^\dagger$ be a marked 2-category. Then, for every object $d\in \DD$,  every marked morphism \begin{tikzcd}
		d \arrow[r,circled] & b
	\end{tikzcd} 
	is an initial object in $L_{\mathcal{W}}(\Nerv_2(\DD_{d\downslash}))$.
\end{lemma}

\begin{proof}
	Consider the identity morphism \begin{tikzcd}
		d \arrow[r,equal] & d
	\end{tikzcd}  and another morphism $\func{f:d\to b}$. Note that the category $\DD_{d\downslash}\left(\id_d,f\right)$ has an initial object given by the diagram 
	\[
	\begin{tikzcd}
	d\arrow[dr,"f"]\arrow[d,equal] & \\
	d\arrow[r,"f"'] & b
	\end{tikzcd}
	\]
	where the 2-morphism filling the diagram is the identity on $f$. Consequently, $\id_d$ is an initial object in the Joyal fibrant replacement of $N_2(\DD_{d\downslash})$. Since localization is cofinal, the image of $\id_d$ is still an initial object in $L_{\mathcal{W}}(N_2(\DD_{d\downslash}))$. We then need only note that, given a marked morphism \begin{tikzcd}
		g:d \arrow[r,circled] & b
	\end{tikzcd} 
	the (strictly commuting) diagram 
	\[
	\begin{tikzcd}
	d\arrow[dr,circled, "g"]\arrow[d,equal] & \\
	d\arrow[r,circled,"g"'] & b
	\end{tikzcd}
	\]
	shows that $g$ is equivalent to $\id_d$ in $L_{\mathcal{W}}(N_2(\DD_{d\downslash}))$, and thus is itself initial. 
\end{proof}

\begin{notation}\glsadd{CCslash}
	Before beginning the proof, we fix some notation for functors which will appear throughout. We consider the functor 
	\[
	\func*{\mathfrak{C}_{\DD\downslash}: \DD\to \Set_\Delta^+; d \mapsto \Nerv_2(\CC_{d\downslash})} 
	\]
	the functor 
	\[
	\func*{\mathfrak{C}_{\CC\downslash}: \CC\to \Set_\Delta^+; c \mapsto \Nerv_2(\CC_{c\downslash})} 
	\]
	and the functor 
	\[
	\func*{\mathfrak{D}_{\DD\downslash}: \DD\to \Set_\Delta^+; d \mapsto \Nerv_2(\DD_{d\downslash}).} 
	\]
	We will denote by, e.g. $\widehat{\mathfrak{C}}_{\DD\downslash}$ the composite $L_{\mathcal{W}}\circ \mathfrak{C}_{\DD\downslash}$ with the fibrant replacement functor (which we also refer to as the localization) on $\Set_\Delta^+$. We will further denote a constant functor $\func{\DD\to \Set_\Delta^+}$ with value $X$ by $\underline{X}$. 
\end{notation}

\begin{construction}
	We note that since $\widehat{\mathfrak{C}}_{\DD\downslash}$ is projectively fibrant, $\func{\widetilde{\cchi}_{\DD}(\widehat{\mathfrak{C}}_{\DD\downslash})\to \Nerv_2(\DD)}$ is a scaled Cartesian fibation. In particular, it is locally Cartesian. By assumption \ref{asump:init}, every fiber of $\func{\widetilde{\cchi}_{\DD}(\widehat{\mathfrak{C}}_{\DD\downslash})\to \Nerv_2(\DD)}$ is non-empty, and has an initial element. We can therefore apply Proposition \ref{prop:locCartTrivKanFib} to find a section 
	\[
	\func{s_F:\Nerv_2(\DD)\to  \widetilde{\cchi}_{\DD}(\widehat{\mathfrak{C}}_{\DD\downslash})} 
	\]
	which is initial in $\on{Map}_{\Nerv_2(\DD)}(\Nerv_2(\DD),\widetilde{\cchi}_{\DD}(\widehat{\mathfrak{C}}_{\DD\downslash}))$. Note that given a marked morphism $g$ in $\DD$, assumption \ref{asump:presundermarked} shows that $s_F(g)$ will be given by a morphism between two initial objects in a fiber, and thus will be an equivalence in $\widetilde{\cchi}_{\DD}(\widehat{\mathfrak{C}}_{\DD\downslash})$. In particular $s_F$ will descend to a functor on $L_{\mathcal{W}}(\DD^\dagger)$. 
\end{construction}

\begin{proof}[Proof (of \autoref{thm:thebigone})]
	Let us briefly outline the proof. The natural transformation $\func{\widehat{\mathfrak{C}}_{\DD\downslash}\nat \underline{L_{\mathcal{W}}(\CC^{\dagger})}}$ induces a map 
	\[
	\begin{tikzcd}
	\gamma:\widetilde{\cchi}_\DD(\widehat{\mathfrak{C}}_{\DD\downslash})\arrow[r] & \widetilde{\cchi}_\DD (\underline{L_{\mathcal{W}}(\CC^{\dagger})})\arrow[r,equal] & \widetilde{\cchi}_\ast(L_{\mathcal{W}}(\CC^{\dagger}))\times \Nerv_2(\DD)\arrow[r] &  \widetilde{\cchi}_\ast(L_{\mathcal{W}}(\CC^{\dagger})).
	\end{tikzcd}
	\]
	We will define $\hat{G}:=\gamma\circ s_F$, and show that $\hat{G}$ defines an inverse to $F_{\mathcal{W}}$. What we mean by this is that, under the identifications induced by the commutative diagram 
	\[
	\begin{tikzcd}
	\Nerv_2(\CC) \arrow[r,"\simeq"]\arrow[d,"F"']&	L_{\mathcal{W}}(\CC^{\dagger})\arrow[d,"F_{\mathcal{W}}"']\arrow[r,"\simeq"] &	\widetilde{\cchi}_\ast(L_{\mathcal{W}}(\CC^{\dagger}))\arrow[d,"\widetilde{\cchi}_\ast(F_{\mathcal{W}})"] \\
	\Nerv_2(\DD) \arrow[r,"\simeq"] &	L_{\mathcal{W}}(\DD^{\dagger})\arrow[r,"\simeq"]&	\widetilde{\cchi}_\ast(L_{\mathcal{W}}(\DD^{\dagger})) 
	\end{tikzcd}
	\]
	The functor $\hat{G}$ provides a lift $\func{L_{\mathcal{W}}(\DD^{\dagger})\to \widetilde{\cchi}_\ast(L_{\mathcal{W}}(\CC^{\dagger}))}$ so that the resulting diagram
	\[
	\begin{tikzcd}
	\Nerv_2(\CC) \arrow[r,"\simeq"]\arrow[d,"F"']&	L_{\mathcal{W}}(\CC^{\dagger})\arrow[r,"\simeq"] &	\widetilde{\cchi}_\ast(L_{\mathcal{W}}(\CC^{\dagger}))\arrow[d,"\widetilde{\cchi}_\ast(F_{\mathcal{W}})"] \\
	\Nerv_2(\DD) \arrow[r,"\simeq"]\arrow[urr,"\hat{G}"]  &	L_{\mathcal{W}}(\DD^{\dagger})\arrow[r,"\simeq"]&	\widetilde{\cchi}_\ast(L_{\mathcal{W}}(\DD^{\dagger})) 
	\end{tikzcd}
	\]
	commutes up to natural equivalence. Since $\hat{G}$ descends to a functor $L_{\mathcal{W}}(\DD^\dagger)\to \tilde{\cchi}_\ast(L_{\mathcal{W}}(\CC^\dagger))$, this immediately implies that $F_{\mathcal{W}}$ is an equivalence. We then note that the bottom horizontal morphism and the top horizontal morphism are equivalent to $\rho_{\DD}\circ s_{\DD}$ and $\rho_{\CC}\circ s_{\CC}$ respectively, by \autoref{cor:sectiongivesID}. We have therefore reduced the problem to showing (1) that $\tilde{\cchi}_\ast(F_{\mathcal{W}})\circ \hat{G}\simeq \rho_{\DD}\circ s_{\DD}$ and (2) that $\hat{G}\circ F\simeq \rho_{\CC}\circ s_{\CC}$. 
	We now embark upon the proof of these facts. 
	\begin{enumerate}
		\item Note that there is a commutative diagram of natural transformations 
		\[
		\begin{tikzcd}
		\mathfrak{C}_{\DD\downslash}\arrow[r,nat]\arrow[d,nat] & \widehat{\mathfrak{C}}_{\DD\downslash}\arrow[d,nat] \arrow[r,nat]\arrow[d,nat]& \underline{L_{\mathcal{W}}(\CC^{\dagger})}\arrow[d,nat, "F_{\mathcal{W}}"]\\ 
		\mathfrak{D}_{\DD\downslash}\arrow[r,nat] & \widehat{\mathfrak{D}}_{\DD\downslash}\arrow[r,nat]& \underline{L_{\mathcal{W}}(\DD^{\dagger})}
		\end{tikzcd}
		\]
		Which induces a commutative diagram 
		\[
		\begin{tikzcd}
		\widetilde{\cchi}_{\DD}(\mathfrak{C}_{\DD\downslash})\arrow[r,]\arrow[d,] & \widetilde{\cchi}_{\DD}(\widehat{\mathfrak{C}}_{\DD\downslash}) \arrow[r,"\gamma"]\arrow[d,"c"]& \widetilde{\cchi}_{\ast}(L_{\mathcal{W}}(\CC^{\dagger}))\arrow[d,"\widetilde{\cchi}_\ast(F_{\mathcal{W}})"]\\ 
		\widetilde{\cchi}_{\DD}(\mathfrak{D}_{\DD\downslash})\arrow[r,"a"]\arrow[rr,bend right, "\rho_{\DD}"'] & \widetilde{\cchi}_{\DD}(\widehat{\mathfrak{D}}_{\DD\downslash})\arrow[r,"b"]&\widetilde{\cchi}_{\ast}(L_{\mathcal{W}}(\DD^{\dagger}))
		\end{tikzcd}
		\]
		
		Applying \autoref{prop:locCartTrivKanFib} and \autoref{lem:IDis2cofinal}, we immediately see that $a\circ s_{\DD}$ is initial in $\on{Map}_{\Nerv_2(\DD)}(\Nerv_2(\DD),\widetilde{\cchi}_{\DD}(\widehat{\mathfrak{C}}_{\DD\downslash}))$. Together, \autoref{prop:locCartTrivKanFib} and \hyperref[asump:init]{condition 1} show that $c\circ s_F$ is another such initial section. Consequently $a\circ s_{\DD}\simeq c\circ s_F$, so  
		\[
		\tilde{\cchi}_\ast(F_{\mathcal{W}})\circ \hat{G}=\tilde{\cchi}_\ast(F_{\mathcal{W}})\circ \gamma\circ s_F=b\circ c\circ s_F\simeq  b\circ a\circ s_{\DD}=\rho_{\DD}\circ s_{\DD}.
		\]
		\item Consider the pullback diagram 
		\[
		\begin{tikzcd}
		\widetilde{\cchi}(F^\ast\mathfrak{C}_{\DD\downslash})\arrow[r]\arrow[d] &  \widetilde{\cchi}(\mathfrak{C}_{\DD\downslash})\arrow[d]\\
		\Nerv_2(\CC)\arrow[r,"F"'] & \Nerv_2(\DD)
		\end{tikzcd}
		\]
		and denote by $\func{\alpha_{F}:N_2(\CC)\to \widetilde{\cchi}(F^\ast\mathfrak{C}_{\DD\downslash})}$ the pullback of the section $s_F$. 
		
		With the aid of the natural transformation $F^\ast\mathfrak{C}_{\DD\downslash}\Rightarrow \underline{N_2(\CC)}$ we get a commutative diagram 
		\[
		\begin{tikzcd}
		\Nerv_2(\CC)\arrow[r,"\alpha_F"]\arrow[d,"F"'] & \widetilde{\cchi}_{\CC}(F^\ast \widehat{\mathfrak{C}}_{\DD\downslash})\arrow[d]\arrow[dr] & \\
		\Nerv_2(\DD)\arrow[r,"s_F"']  &\widetilde{\cchi}_{\DD}(\widehat{\mathfrak{C}}_{\DD\downslash})\arrow[r,"\gamma"'] & \widetilde{\cchi}_\ast(L_{\mathcal{W}} (\CC^{\dagger}))
		\end{tikzcd}
		\]
		It will therefore suffice to show that the top composite is equivalent to $\rho_{\CC}\circ s_{\CC}$. We therefore consider the commutative diagram 
		\[
		\begin{tikzcd}
		\widetilde{\cchi}_{\CC}(\mathfrak{C}_{\CC\downslash})\arrow[r,"a"]\arrow[dr] & \widetilde{\cchi}_{\CC}(\widehat{\mathfrak{C}}_{\CC\downslash})\arrow[d,"c"]\arrow[dr] & \\
		& \widetilde{\cchi}_{\CC}(F^\ast \widehat{\mathfrak{C}}_{\DD\downslash})\arrow[r,"b"'] &\widetilde{\cchi}_\ast(L_{\mathcal{W}} (\CC^{\dagger})) 
		\end{tikzcd}
		\]
		and note that the top composite is $\rho_{\CC}$. By analogous reasoning to that above (now using assumption \ref{assump:markedinit} as well), both $c\circ a\circ s_{\CC}$ and $\alpha_F$ are initial objects in the $\infty$-category of sections of $\widetilde{\cchi}_{\CC}(F^\ast \widehat{\mathfrak{C}}_{\DD\downslash})$. Consequently, we find that the composite of $b$ with $\alpha_F$ is equivalent to the composite of the top path with $s_{\CC}$, i.e. $\rho_{\CC}\circ s_{\CC}$. Therefore, we find that $\hat{G}\circ F=\gamma\circ s_F\circ F\simeq \rho_{\CC}\circ s_{\CC}$. \qedhere
	\end{enumerate}
\end{proof}
\section{Corollaries and applications}
\label{sec:coraps}

The primary purpose of this section is twofold. First, we derive a more computationally tractable condition from \autoref{thm:thebigone}, and second, we provide brief r\'esum\'e of existing results which form special cases of \autoref{thm:thebigone}.

We begin with a corollary using an apparently stronger criterion, which we discussed in the introduction as \autoref{thm:AdagIntro}:

\begin{corollary}\label{cor:equivtobig}
	Let $F:\CC^\dagger \to \DD^\dagger$ be a functor of 2-categories with weak equivalences. Suppose that, for every object $d\in \DD$, 
	\begin{enumerate}
		\item There exists an object $c\in \CC$ and a marked morphism \begin{tikzcd}
			d \arrow[r,circled] & F(c)
		\end{tikzcd}
		\item Every marked morphism \begin{tikzcd}
			d \arrow[r,circled] & F(c)
		\end{tikzcd} 
		is initial in the localization $L_\mathcal{W}(\Nerv_2(\CC_{d\downslash})^\dagger)$. 
	\end{enumerate}
	Then the induced functor $\func{F_{\mathcal{W}}: L_\mathcal{W}(\Nerv_2(\CC^\dagger))\to L_\mathcal{W}(\Nerv_2(\DD^\dagger))}$ is an equivalence of $\infty$-categories. 
\end{corollary}

\begin{proof}
	Since, by assumption, marked morphisms are initial and marked-ness is stable under composition (since we are dealing with 2-categories with weak equivalences), conditions \ref{asump:init}, \ref{assump:markedinit},  and \ref{asump:presundermarked} of \autoref{thm:thebigone} are all immediate. The corollary then follows.  
\end{proof}

\begin{remark}
	We say that this criterion is `apparently stronger' because in fact \autoref{cor:equivtobig} is equivalent to \autoref{thm:thebigone}. To see the reverse implication, let $\func{F:\CC^\dagger \to \DD^\dagger}$ be a functor of marked 2-categories satisfying the criteria of \autoref{thm:thebigone}.
	
	Since $g_d$ becomes an equivalence in $L_{\mathcal{W}}(\DD^\dagger)$, we can add $\{g_d\}_{d\in \DD}$ to the set of marked morphisms of $\DD$ without changing the localization. Since precomposition with $\func{g_d:d\to F(c)}$ sends the initial object $\id_{F(c)}$ to $g_d$, we also note that precompositon with $g_d$ preserves initial objects in the localized slices. More generally, we can close under composition without changing the localizations, so we set $\CC^s:=Q(\CC^\dagger)$ and $\DD^s:=Q(\DD,\mathcal{W}_\DD\cup \{g_d\}_{d\in \DD})$.  We will show that $\CC^s$ and $\DD^s$ satisfy the hypotheses of \autoref{cor:equivtobig}. 
	\begin{enumerate}
		\item Note that $\func{L_{\mathcal{W}}(\CC^\dagger_{d\downslash})\to L_{\mathcal{W}}(\CC^s_{d\downslash})}$ is a localization map, so every initial object in the former is also initial in the latter. In particular, for every $c\in \CC$, the object $\id_{F(c)}$ is inital in $L_{\mathcal{W}}(\CC^s_{F(c)\downslash})$.
		
		\item If we let \begin{tikzcd}
			f:d\arrow[r,circled] & F(c)
		\end{tikzcd} be a marked morphism in $\DD^s$, it is immediate from the construction that $\func{f^\ast:L_{\mathcal{W}}(\CC^\dagger_{d\downslash})\to L_{\mathcal{W}}(\CC^\dagger_{b\downslash})}$ preserves initial objects. A brief consideration of the commutative diagram 
		\[
		\begin{tikzcd}
		L_{\mathcal{W}}(\CC^\dagger_{d\downslash})\arrow[r]\arrow[d,"f^\ast"'] & L_{\mathcal{W}}(\CC^s_{d\downslash})\arrow[d,"f^\ast"]\\
		L_{\mathcal{W}}(\CC^\dagger_{b\downslash})\arrow[r] & L_{\mathcal{W}}(\CC^s_{b\downslash})
		\end{tikzcd}
		\]
		then shows that $\func{f^\ast:L_{\mathcal{W}}(\CC^s_{d\downslash})\to L_{\mathcal{W}}(\CC^s_{b\downslash})}$ preserves initial objects. 
	\end{enumerate}
	
	The hypotheses of \autoref{cor:equivtobig} immediately follow. For every $d\in \DD$, there is a marked morphism \begin{tikzcd}
		g_d: d \arrow[r,circled] & F(c)
	\end{tikzcd} 
	which shows the first hypothesis is satisfied. Moreover, for any marked morphism \begin{tikzcd}
		f: d \arrow[r,circled] & F(c)
	\end{tikzcd} in $\DD^s$, the fact that $f^\ast$ preserves initial objects, and $\id_{F(c)}$ is initial shows that $f=\id_{F(c)}\circ f$ is initial in $L_{\mathcal{W}}(\CC^s_{d\downslash})$, which shows the second hypothesis to be satisfied.
\end{remark}

Having derived this corollary, we turn to the ways in which \autoref{cor:equivtobig} generalizes other results in the literature. We begin with a proposition of Bullejos and Cegarra (from \cite{Bullejos_QuillenA}), the criterion of which is closest in spirit to ours.  

\begin{proposition}[Bullejos and Cegarra]\label{prop:BullejosRealization}
	Let $\func{F:\CC\to \DD}$ be a 2-functor. Suppose that, for every $d\in \DD$, there is a homotopy equivalence $|\Nerv_2(\CC_{d\downslash})|\simeq \ast$. Then $\func{|F|:|\Nerv_2(\CC)|\to |\Nerv_2(\DD)|}$ is a homotopy equivalence. 
\end{proposition}

\begin{proof}
	We view $\CC$ and $\DD$ as having every 1-morphism marked, so that the localizations $L_{\mathcal{W}}(N_2(\CC))$ and $L_{\mathcal{W}}(N_2(\DD))$ coincide with geometric realizations $|\Nerv_2(\CC)|$ and $|\Nerv_2(\DD)|$, respectively. Since the slice 2-categories $\CC_{d\downslash}$ are non-empty and contractible it is immediate that both of the criteria of \autoref{cor:equivtobig} are fulfilled. The proposition follows. 
\end{proof}

\begin{remark}
	The proof of \autoref{prop:BullejosRealization} presented in \cite{Bullejos_QuillenA} uses more classical techniques --- like the yoga of bisimplicial sets --- to obtain the result. In this sense, their proof is not dissimilar to Quillen's original proof of Theorem A. Much as Theorem A has been proved and reproved using new layers of technology (see, e.g. \cite{HTT}), \autoref{prop:BullejosRealization} can be approached either from classical homotopy-theoretic or from $\infty$-categorical viewpoints.
\end{remark}

We now turn to an $\infty$-categorical proposition of Walde (from \cite{Walde}), which at first blush resembles \autoref{thm:thebigone} rather less. It will turn out that this is again a special case of \autoref{cor:equivtobig}, and provides an example of a special case in which the criteria are easier to check 1-categorically. We first develop a bit of notation, following \cite{Walde}

\begin{definition}
	Let $\func{F:\mathcal{C}\to \mathcal{D}}$ be a functor of 1-categories. For each $d\in \mathcal{C}$, we define the \emph{weak fiber}  $\mathcal{C}_d\subset \mathcal{C}_{d/}$ to be the full subcategory on the isomorphisms $\func{d\to[\simeq] f(c)}$. We similarly definte $\mathcal{C}^d\subset \mathcal{C}_{/d}$ to be the full subcategory on the isomorphisms. 
\end{definition}

\begin{proposition}\label{prop:GeneralWalde}
	Let $\func{F:\mathcal{C}\to \mathcal{D}}$ be a functor of 1-categories. Suppose that for each $d\in \mathcal{D}$
	\begin{enumerate}
		\item There is a contractible subcategory $\mathcal{B}_d\subset \mathcal{C}_d$ 
		\item The inclusion $\func{i_d:\mathcal{B}_d\to \mathcal{C}_{d/}}$ is coinitial. 
	\end{enumerate}
	Let $\mathcal{W}\subset \mathcal{C}$ be the wide subcategory of morphisms $f$ such that $F(f)$ is an isomorphism. Then 
	\[
	\func{F_\mathcal{W}:L_{\mathcal{W}}(\mathcal{C})\to \mathcal{D}} 
	\]
	is an equivalence of $\infty$-categories. 
\end{proposition}

\begin{proof}
	We consider $\mathcal{C}$ and $\mathcal{D}$ as categories with weak equivalences, where the marking on $\mathcal{C}$ is given by $\mathcal{W}$, and the marking on $\mathcal{D}$ consists of the isomorphisms. By the first hypothesis, there is at least one isomorphism $\func{d\to[\simeq]F(c)}$ for each $d\in \mathcal{C}$, so the first condition in \autoref{cor:equivtobig} is satisfied. 
	
	Since $\mathcal{B}_d$ is contractible and $i_d$ is coinitial, the induced functor on localizations $i_d: L_{\mathcal{W}}(\mathcal{B}_d)\to L_{\mathcal{W}}(\mathcal{C}_{d/})$ is coinitial. It thus follows that the elements of $\mathcal{B}_d$ are initial in $L_{\mathcal{W}}(\mathcal{C}_{d/})$. Moreover, since $i_d$ is coinitial, given an isomorphism $\func{f:d\to[\simeq]F(c)}$, the slice category $(\mathcal{B}_d)_{/f}$ is non-empty, i.e. we have a commutative diagram 
	\[
	\begin{tikzcd}
	& d\arrow[dr,"f"]\arrow[dl,"g"'] & \\
	F(b)\arrow[rr,"F(h)"'] & &F(c)
	\end{tikzcd}
	\] 
	where $g\in \mathcal{B}_d$. By 2-out-of-3, this means that $F(h)$ is an isomorphism in $\mathcal{D}$, and thus $h\in \mathcal{W}$. Therefore, $g$ and $f$ are equivalent in $L_{\mathcal{W}}(\mathcal{C}_{d/})$, and thus $f$ is initial. 
	
	The criteria of \autoref{cor:equivtobig} are thus satisfied, and the conclusion follows immediately. 
\end{proof}

\begin{corollary}[Walde]
	Let $\func{F:\mathcal{C}\to \mathcal{D}}$ be a functor of 1-categories. Suppose that for each $d\in \mathcal{D}$
	\begin{enumerate}
		\item There is a subcategory $\mathcal{B}_d\subset \mathcal{C}^d$ and an initial object $f_d\in \mathcal{B}_d$. 
		\item The inclusion $\func{i_d:\mathcal{B}_d\to \mathcal{C}_{/d}}$ is cofinal. 
	\end{enumerate}
	Let $\mathcal{W}\subset \mathcal{C}$ be the wide subcategory of morphisms $f$ such that $F(f)$ is an isomorphism. Then 
	\[
	\func{F_\mathcal{W}:L_{\mathcal{W}}(\mathcal{C})\to \mathcal{D}} 
	\]
	is an equivalence of $\infty$-categories. 
\end{corollary}

\begin{proof}
	This follows immediately from \autoref{prop:GeneralWalde} applied to $F^\op:\mathcal{C}^\op\to \mathcal{D}^\op$. 
\end{proof}

\begin{remark}
	The criteria of \autoref{prop:GeneralWalde} and those Walde are of particular interest despite being less general for several reasons. Firstly, they provide conditions under which a 1-category can be viewed as an $\infty$-categorical localization of another category. Walde uses the proposition to give a model for invertible cyclic $\infty$-operads in terms of $\infty$-functors out of a 1-category. Secondly, these conditions lend themselves more easily to direct computation.
\end{remark}

\section{The cofinality conjecture}\label{sec:cofinality}

In this section we develop a basic theory of marked colimits of 2-categories, providing a natural framework  to interpret \autoref{thm:thebigone} as a cofinality statement. The larger context of Quillen's Theorem A is that two important properties coincide (though this is no coincidence): the criterion that a functor $F$ must satisfy in Quillen's Theorem A is \emph{equivalent} to precomposition with $F$ preserving $(\infty,1)$-colimits. Moreover, the truncation of the criterion is equivalent to precomposition with $F$ preserving strict 1-categorical colimits. 

Here, we provide an analogue of the latter statement one rung up the ladder of categorification. We show that an appropriately decategorified version of the criterion for Theorem A\textsuperscript{$\dagger$} is equivalent to precomposition with $F$ preserving marked colimits of 2-categories. 

If we view the framework discussed in this section as a reflection of a possible theory of marked $(\infty,2)$-colimits which truncates to the 2-categorical theory, this statement may be taken as evidence that the dual nature of Quillen's criterion generalizes to the 2-categorical case. The precise formulation of this is the \emph{cofinality conjecture} developed at the end of the section.

\subsection{Marked colimits}
We begin with the key definitions and properties necessary to work with marked 2-colimits.

\begin{definition}
	Let $\CC^{\dagger}$ be a marked 2-category. We define a 2-functor
	\[
		\func*{\hat{\mathfrak{C}}^{\dagger}_{\CC \upslash}: \CC^{(\op,\mathblank)} \to \Cat; c \mapsto \on{ho}\left(\CC_{c \upslash}\right)^{\dagger}[W^{-1}]}
	\] 
\end{definition}

\begin{definition}
	Let $\CC^{\dagger}$ be a marked 2-category. Given a 2-category $\mathbb{A}$ and a  2-functor $\func{F: \CC \to \mathbb{A}}$, we define 
	\[
		\func*{N^{\dagger}_{F}: \mathbb{A} \to \Cat; a \mapsto \on{Nat}\left(\hat{\mathfrak{C}}^{\dagger}_{\CC \upslash},\mathbb{A}(F(\mathblank),a)\right)}
	\]
\end{definition}

\begin{definition}\label{def:markedcone}	
	We call a natural transformation $\func{\hat{\mathfrak{C}}^{\dagger}_{\CC \upslash} \nat \mathbb{A}(F(\mathblank),a)}$ a \emph{marked cocone} for $F$. It is easy to check that such natural transformation is uniquely determined by the following data:
	\begin{itemize}[noitemsep]
		\item An object $a \in \mathbb{A}$.
		\item For every object $c \in \CC^{\dagger}$ a 1-morphism $\func{\alpha_c: F(c) \to d}$.
		\item For every morphism $\func{c \to[u] c'}$ in $\CC^{\dagger}$ a 2-morphism $\func{\alpha_u:\alpha_c \nat \alpha_{c'}\circ F(u)}$.
	\end{itemize}
	These data are subject to the conditions:
	\begin{enumerate}
		\item\label{2comma:invertmarked} For every marked morphism $u$ in $\CC^{\dagger}$, the 2-morphism $\alpha_{u}$ is invertible.
		\item\label{2comma:coherent1} Given a pair of 1-morphisms $\func{c \to[u] c' \to[v] c''}$, the composite
		\[
		\begin{tikzcd}[ampersand replacement=\&]
		\alpha_c \arrow[r,nat,"\alpha_u"] \& \alpha_{c'}\circ F(u) \arrow[r,nat,"\alpha_{v}*F(u)"] \& \alpha_{c''}\circ F(vu) 
		\end{tikzcd}
		\]
		equals $\alpha_{vu}$ and similarly we have that $\alpha_{id}$ is the identity 2-cell.
		\item\label{2comma:coherent2} Given a 2-morphism $\func{\beta: u \nat w}$, the 2-morphism $\func{\alpha_c \nat[\alpha_u] \alpha_{c'}\circ F(u) \nat[\alpha_{c'}*\beta] \alpha_{c'}\circ F(w)}$ equals $\alpha_{w}$.
	\end{enumerate}
\end{definition}

\begin{remark}
	It is worth noting that this definition is effectively the same as that of \cite{Dubuc_limits}. In the minimally and maximally marked cases it specializes to the notion of pseudocolimit and lax colimit, respectively.  
\end{remark}

\begin{definition}\glsadd{El}
	Let $\mathbb{A}^{\dagger}$ be a marked 2-category and consider a 2-functor
	\[
	\func{F: \mathbb{A} \to \Cat}
	\]
	We define a marked 2-category $\operatorname{El}(F)^{\dagger}$, as follows:
	\begin{itemize}[noitemsep]
		\item Objects are pairs $(a,x)$  where $a  \in \mathbb{A}$ and $x \in F(a)$.
		\item 1-morphisms from $(a,x)$ to   $(a',y)$  are given by pairs $(u,\varphi)$ where $\func{a  \to[u] a'}$ in $\mathbb{A}$ and $\func{\varphi:F(u)x \to y}$ in $F(a')$.
		\item 2-morphisms $\func{(u,\varphi) \nat (v,\psi)}$ are given by a 2-morphism $\func{u \nat[\theta] v}$ in $\mathbb{A}$ making the following diagram commute
		\[
		\begin{tikzcd}[ampersand replacement=\&]
		F(u)x \arrow[dr] \arrow[rr,"F(\theta)_x"] \& \& F(v)x \arrow[dl]\\
		\&  y \& 
		\end{tikzcd}
		\]
	\end{itemize}
	We equip $\operatorname{El}(F)^{\dagger}$ with a marking by  declaring $(u,\varphi)$ to be marked if and only if $u$  is marked and $\varphi$ is an isomorphism.
\end{definition}

\begin{remark}\label{rem:scaledcocartesian}
	This is just a special case of the 2-categorical Grothendieck construction of \cite{Buckley} decorated with a marking. Observe that the edges marked are locally coCartesian and that the map $\func{\on{El}(F)^{\dagger} \to \mathbb{A}}$ is a scaled coCartesian fibration (cf. Proposition 6.9 and Lemma 6.10 from \cite{Harpaz_Quillen}).
\end{remark}

\begin{definition}
	We define the marked 2-category, $\mathbb{A}_{F \upslash}^{\dagger}$ of marked cocones for $F$ as $\on{El}(N^{\dagger}_{F})$. Unraveling the definitions we obtain the following description:
	\begin{itemize}[noitemsep]
		\item Objects are given by marked cones.
		\item Given two marked cones, $\set{\func{\alpha_c: F(c) \to a}}_{c \in \CC}$, $\set{\func{\beta_c: F(c) \to a'}}_{c \in \CC}$ we define a morphism $\func{\set{\alpha_c}_{c \in \CC} \to \set{\beta_c}_{c \in \CC}}$ to be the data of a 1-morphism $\func{\theta: a \to a'}$ and a family of 2-morphisms,
		\[
		\set{\func{\epsilon_c:\theta \circ \alpha_c \nat \beta_c }}_{c \in \CC},
		\]
		such that the diagram
		\[
		\begin{tikzcd}[ampersand replacement=\&]
		\theta \circ \alpha_c \arrow[d,nat] \arrow[r, nat] \& \beta_c  \arrow[d,nat]\\
		\theta \circ  \alpha_{c'} \circ F(u) \arrow[r,nat] \& \beta_{c'}\circ F(u)
		\end{tikzcd}
		\]
		commutes for every morphism $u$ in $\CC$.
		\item Given two morphisms $\set{\func{\epsilon_c,\theta}}_{c \in \CC}, \set{\func{\eta_c,\gamma}}_{c \in \CC}$ we define a 2-morphism to be the data of a 2-morphism $\func{ \theta \nat \gamma}$ making the  diagram
		\[
		\begin{tikzcd}[ampersand replacement=\&]
		\theta \circ \alpha_c \arrow[d,nat] \arrow[r,nat, "\epsilon_c"] \& \beta_{c} \\
		\gamma \circ \alpha_c \arrow[ur,nat,"\eta_c"'] \& .
		\end{tikzcd}
		\]
		commute
	\end{itemize}
	We declare a morphism to be marked if the 2-morphisms $\epsilon_c$ are all invertible.
\end{definition}

\begin{definition}\glsadd{colimdag}
	Given a marked category $\mathbb{C}^{\dagger}$ and a 2-functor $\func{F: \mathbb{C} \to \mathbb{A}}$, we say that $a \in \mathbb{A}$ is the \emph{marked colimit} of $F$ if there exists a natural transformation $\func{\mathbb{A}(a,\mathblank) \nat N^{\dagger}_{F}}$ which is a levelwise equivalence of categories. We will denote the marked colimit of $F$ by $\on{colim^{\dagger}}F$.
\end{definition}

\begin{example}\label{exam:adjiscolim}
	Let us equip $[2]$ with a marking by declaring all edges  except $\func{1 \to 2}$ to be marked. We will denote this marked category by $[2]^{\diamond}$. Consider a pair of adjoint functors
	\[
	L: \mathcal{C} \llra \mathcal{D}:R
	\]
	such that $L\circ R= \on{id}_D$ is the counit of the adjunction. Let us define a functor
	\[
	\func{T: [2] \to \Cat}
	\]
	by mapping $T(0)=\mathcal{D}$, $T(1)=\mathcal{C}$ and $T(2)=\mathcal{D}$. We will denote the morphisms of $[2]$ by its source and target. Then, we define $T(01)=R$ and $T(12)=L$. We will show that $\on{colim^{\diamond}}T \simeq \mathcal{C}$.
	
	To construct a marked cone for $T$ we set $\alpha_0=R$, $\alpha_1=\on{id}_C$ and $\alpha_{2}=R$. We set $\alpha_{01}=\on{id}_R$, $\alpha_{12}=\eta$ (the unit of the adjunction) and $T(02)=\on{id}_R$. Note that the triangle identities imply that this is indeed a marked cone. Using the 2-categorical Yoneda lemma (cf. e.g. Chapter 8 of \cite{Johnson_Yau}) we obtain a natural transformation 
	\[
		\func{\scr{R}:\Cat(\mathcal{C},\mathblank) \nat N^{\diamond}_{T}.}
	\]
	
	  We check that $\scr{R}$ is a levelwise equivalence of categories. Let $\scr{X} \in \Cat$, and consider the functor
	 \[
	 \func{\scr{R}_{\scr{X}}: \Cat(\mathcal{C},\scr{X}) \to \on{Nat}\left(\widehat{[\frak{2}]}^{\diamond}_{[2]\upslash},\Cat(T(\mathblank),\scr{X})\right). }
	 \]
	 Unwinding the definitions, it easy to check that $\scr{R}_{\scr{X}}$ is fully faithful. To show essential surjectivity let $\set{\beta_i}_{i \in [2]}$ be a marked cone with tip $\scr{X}$. We construct a morphism of marked cones $\set{\func{\epsilon_i:\beta_1 \circ \alpha_i \nat \beta_i}}_{i \in [2]}$ as follows: First we note that we have by assumption an invertible 2-morphism
	\[
	\begin{tikzcd}[ampersand replacement=\&]
	\beta_{01}: \beta_0 \arrow[r,nat,"\simeq"] \& \beta_1 \circ R= \beta_1 \circ \alpha_0
	\end{tikzcd}
	\]
	So we set $\epsilon_0=(\beta_{01})^{-1}$. It is clear that we can set $\epsilon_1=\on{id}_{\beta_1}$. Let us consider the 2-morphism
	\[
	\func{\beta_{12}:\beta_1 \nat \beta_2 \circ L}
	\]
	then whiskering with $R$ we obtain a 2-morphism $\func{\beta_{12}*R: \beta_1 \circ R \nat \beta_2}$. Since the composite
	\[
	\func{\beta_0 \nat \beta_1 \circ R \nat \beta_2}
	\]
	is invertible, so is $(\beta_{12}*R)$ and we set $\epsilon_2= (\beta_{12}*R)$. The only nontrivial verification left to do in order to show that we have defined a morphism of marked cones its to exhibit commutativity of the following diagram
	\[
	\begin{tikzcd}[ampersand replacement=\&]
	\beta_1 \circ \alpha_1 \arrow[r,nat] \arrow[d,nat] \& \beta_1 \arrow[d,nat]\\
	\beta_1 \circ \alpha_2 \circ L \arrow[r,nat] \& \beta_2 \circ L
	\end{tikzcd}
	\]
	which follows immediately from the triangle identities. Since the 2-morphisms $\epsilon_i$ are all invertible we conclude that we have defined a marked edge in the category of marked cones. This finishes the proof. 
\end{example}

\begin{theorem}\label{thm:buckleycomputes}
	Let $\CC^{\dagger}$ be a marked 2-category and consider a 2-functor
	\[
	\func{F: \CC \to \Cat.}
	\]
	Equip $\ho\left(\operatorname{El}(F)^{\dagger}\right)$ with the induced marking and denote its 1-categorical localization by $\ho \left(\operatorname{El}(F)\right)[W^{-1}]$. Then, we have an equivalence of categories $\operatorname{colim}^{\dagger}F \isom \ho\left(\operatorname{El}(F)\right)[W^{-1}]$.
\end{theorem}
\begin{proof}
	We define functors $\func{\alpha_c: F(c) \to \ho\left(\operatorname{El}(F)\right)[W^{-1}]}$  sending an object $x$ to $(c,x)$ and sending a morphism $\func{x \to[\varphi] y}$ to the pair $(\id_c,\varphi)$. Given a 1-morphism $\func{u: c \to c'}$ in $\CC$ we construct a natural transformation $\func{\alpha_u: \alpha_c \nat \alpha_{c'}\circ F(u)}$ whose component at $x$ is given by 
	\[
	\func{(u,\id_{F(u)x}): (c,x) \to (c',F(u)x)}.
	\]
	We now check that the conditions of \autoref{def:markedcone} are satisfied, thus defining a marked cocone. Note that the components of $\alpha_u$ are marked whenever $u$ is marked morphism in $\CC$. This implies that condition \ref{2comma:invertmarked} is satisfied. Condition \ref{2comma:coherent1} is satisfied immediately by construction. Finally to show that this family of natural transformations is compatible with the 2-morphisms of $\CC$ (condition \ref{2comma:coherent2}), observe that given $\func{u \nat v}$ in $\CC$ the components of the natural transformation
	\[
	\func{\alpha_c \nat[\alpha_u] \alpha_{c'}\circ F(u) \nat \alpha_{c'}\circ F(v)}
	\]
	are related to the components of $\alpha_{v}$ by a 2-morphism in $\operatorname{El}(F)$ and therefore become equal in the homotopy category. It is a straightforward exercise to check that this cone represents the functor $N^{\dagger}_{F}$.
\end{proof}

\begin{example}
	We return again to the setting of \autoref{exam:adjiscolim}, considering the functor 
	\[
	\func{T:[2]\to \Cat}
	\]
	which picks out the adjoint functors $L$ and $R$, together with the composite $L\circ R=\id_{\mathcal{D}}$. Define $\on{El}(T)_{\leq 1}\subset \on{El}(T)$ to be the full subcategory on the objects over $0$ or $1$ (i.e. the Grothendieck construction of the functor picking out $R$ as a morphism in $\Cat$), and equip it with the induced marking $\on{El}(T)_{\leq 1}^\dagger$. 
	
	Since both markings are composition closed, we can apply the dual of \autoref{cor:equivtobig} to the inclusion 
	\[
	\func{I:\on{El}(T)_{\leq 1}^\dagger\to \on{El}(T)^\dagger.}
	\]
	The criteria are trivially satisfied for objects over $0$ or $1$. In the case of an object $(2,d)\in\on{El}(T)^\dagger$, it is immediate that the morphism $(1,R(d))\to (2,d)$ given by $\id_d$ is terminal. For a marked morphism $\phi:(0,b)\to (2,d)$, the diagram 
	\[
	\begin{tikzcd}
	(0,b)\arrow[rr,circled, "R(\phi)"] \arrow[dr,circled,"\phi"']& &  (1,R(d))\arrow[dl,"\id_d"]\\
	& (2,d) & 
	\end{tikzcd}
	\]
	commutes.  It follows that every marked morphism in equivalent to $\id_d$ in the slice category, and thus is terminal. Therefore $I$ induces an equivalence on $\infty$-localizations. 
	
	One can then easily check that the obvious functors 
	\[
	\func{\mathcal{C}^\natural \to[\iota] \on{El}(T)_{\leq1}^\dagger \to[r] \mathcal{C}^\natural} 
	\]
	satisfy $r\circ \iota=\id_{\mathcal{C}}$, and that there is a marked natural transformation $\func{\id_{\on{El}(T)_{\leq1}^\dagger}\nat \iota\circ r}$, so that $r$ and $\iota$ induce an equivalence on $\infty$-categorical localizations. One can also show this by applying the undualized version of \autoref{cor:equivtobig} to $\func{\iota:\mathcal{C}^\natural \to \on{El}(T)_{\leq1}^\dagger}$.

	
	It is worth noting that it is no accident that this is an equivalence of $\infty$-categorical localizations, rather than just a 1-categorical one. In a hypothetical $(\infty,2)$-categorical theory of marked colimits, the adjunction data should still give the $(\infty,2)$-marked colimit of $T$. We would thus expect the Grothendieck construction, once $\infty$-localized, to compute the $(\infty,2)$-marked colimit, to be equivalent to $\mathcal{C}$
\end{example}

\begin{remark}
	Let $C^{\dagger}$ be a marked 1-category and consider the constant functor 
	\[
	\func*{\underline{*}: C \to \Cat}
	\]
	with value the terminal category $*$. Then \autoref{thm:buckleycomputes} shows that its 1-categorical localization $C^\dagger[W^{-1}]$ is the marked colimit of the functor $\underline{*}$.
\end{remark}

\subsection{A criterion for cofinality}

We now come to the decategorified condition and the discussion of (strict) marked cofinality.

\begin{definition}
	Let $\func{f: \CC^{\dagger} \to \DD^{\dagger}}$ be a marked 2-functor. We say that $f$ is \emph{marked cofinal} if and only if, for every 2-functor $\func{F:\DD \to \mathbb{A}}$, the canonical restriction map
	\[
	\func{N^{\dagger}_{F} \nat N^{\dagger}_{f^{*}F}}
	\]
	is a levelwise equivalence of categories.
\end{definition}	

\begin{theorem}\label{thm:decatAdag}
	Let $\func{f: \CC^{\dagger} \to \DD^{\dagger}}$ be a marked 2-functor. Then, $f$ is cofinal if and only if the following conditions hold
	\begin{enumerate}
		\item\label{asump:initstrict} For every object $d\in \DD$, there exists a morphism $\func{g_d: d\to f(c)}$ which is initial in both $\ho \left(\CC_{d\upslash}^\dagger\right)[\mathcal{W}^{-1}]$ and  $\ho \left(\DD_{d\upslash}^\dagger\right)[\mathcal{W}^{-1}]$.
		\item\label{assump:markedinitstrict} Every marked morphism \begin{tikzcd}
			d \arrow[r,circled] & f(c)
		\end{tikzcd}
		is initial in $\ho \left(\CC_{d\upslash}^\dagger\right)[\mathcal{W}^{-1}]$. 
		\item\label{asump:presundermarkedstrict}  For any marked morphism \begin{tikzcd}f:b\arrow[r,circled] & d\end{tikzcd} in $\DD$, the induced functors 
		\[
		\func{f^\ast: \ho \left(\CC_{d\upslash}^\dagger\right)[\mathcal{W}^{-1}] \to \ho \left(\DD_{d\upslash}^\dagger\right)[\mathcal{W}^{-1}]}
		\] 
		preserve initial objects.
	\end{enumerate}
\end{theorem}
\begin{proof}
	\begin{implications}
		\implication Let $d \in \DD$, and define the functor
		\[
		\func*{R_d: \DD \to \Cat;d^{\prime} \mapsto \DD (d,d^{\prime})}
		\]
		It is immediate that $\on{El}(R_d)=\DD_{d \upslash}$ (resp. $\on{El}(f^*R_d)=\CC_{d \upslash}$). Invoking \autoref{thm:buckleycomputes} and the fact that $f$ is cofinal, we obtain an equivalence of categories
		\[
		\func{f_d: \ho\left(\CC_{d\upslash}^{\dagger}\right)[W^{-1}] \to \ho\left(\DD_{d\upslash}^{\dagger}\right)[W^{-1}].}
		\]
		Using the fact that $f_d$ is essentially surjective we can find some $\func{g_d:d \to f(c)}$ in $\ho\left(\CC_{d\upslash}^{\dagger}\right)[W^{-1}]$ that becomes equivalent to the identity map in $ \ho\left(\DD_{d\upslash}^{\dagger}\right)[W^{-1}]$. It is clear that $g_d$ is initial in both categories since equivalences of categories preserve and reflect initial objects. Since marked morphisms are initial in $\ho \left(\DD_{d\upslash}^\dagger\right)[\mathcal{W}^{-1}]$ the second condition is immediate from the fact that $f_d$ is an equivalence. 
		
		Let $\begin{tikzcd}
		b \arrow[r,circled,"u"] & d
		\end{tikzcd}$
		be a marked morphism and consider the following commutative diagram induced by precomposition with $u$
		\[
		\begin{tikzcd}[ampersand replacement=\&]
		\ho\left(\CC_{d\upslash}^{\dagger}\right)[W^{-1}] \arrow[r,"f_d"] \arrow[d] \& \ho\left(\DD_{d\upslash}^{\dagger}\right)[W^{-1}] \arrow[d] \\
		\ho\left(\CC_{b\upslash}^{\dagger}\right)[W^{-1}] \arrow[r,"f_b"] \& \ho\left(\DD_{b\upslash}^{\dagger}\right)[W^{-1}]
		\end{tikzcd}
		\]
		Then it is clear that $u \sim u \circ g_d$. This implies that $u \circ g_d$ is initial in both categories, so the third condition holds. 
		\backimplication Let $\func{F: \DD \to \mathbb{A}}$ be a 2-functor. We will show that the restriction functor
		\[
		\func{f^{*}:\mathbb{A}^{\dagger}_{F\upslash} \to \mathbb{A}^{\dagger}_{f^{*}F\upslash}}
		\]
		is an equivalence of categories after passages to fibers. We fix once and for all a choice of morphisms 
		\[
		\set{\gamma_d: \begin{tikzcd}[ampersand replacement=\&]
			d \arrow[r] \& f(c_d)
			\end{tikzcd}}_{d \in \DD}
		\] 
		satisfying the conditions of the theorem. Let $\set{\alpha_c}_{c \in \CC}$ be a marked cocone for $f^{*}F$. For every $d \in \DD$ we define a functor
		\[
		\func*{\hat{T}^{\alpha}_d: \CC^{\dagger}_{d\upslash} \to \mathbb{A}(F(d),a)}
		\]
		by sending an object $\func{d \to f(c)}$ to the composite $\func{F(d)\to F(f(c)) \to[\alpha_c] a}$. Given a morphism  in  $\CC^{\dagger}_{d\upslash}$ we map  to  the 2-morphism obtained by pasting  the  diagram  below
		\[
		\begin{tikzcd}[row sep =3em, column sep =4em]
		& |[alias=X]| F(f(c))\arrow[r, "\alpha_c", ""{name=V,inner sep=1pt,below}]\arrow[d] & a \\
		F(d)\arrow[ur]\arrow[r,""{name=U,inner sep=1pt,below left}] & |[alias=Y]|F(f(c^{\prime}))\arrow[ur,"\alpha_{c^\prime}"'] &
		\arrow[Rightarrow, from=X,
		to=U, shorten <=.3cm,shorten >=.6cm, "\theta"']\arrow[Rightarrow, from=V,
		to=Y, shorten <=.3cm,shorten >=.6cm, "\alpha_u"']
		\end{tikzcd}
		\]
		Let  us  note  that if two 1-morphisms $u,v$ in  $ \CC^{\dagger}_{d\upslash}$  are related by a 2-morphism then 
		\[
		\hat{T}^{\alpha}_d(u)=\hat{T}^{\alpha}_d(v).
		\]
		Moreover, one  can immediately check that $\hat{T}^{\alpha}_d$ maps marked  edges to  invertible 2-morphisms.  Therefore, we obtain a  factorization
		\[
		\func{T^{\alpha}_d: \ho \left( \CC^{\dagger}_{d\upslash}\right)[W^{-1}] \to \mathbb{A}(F(d),a).}
		\]
		Finally, we observe that for any  $\func{w:d \to d^{\prime}}$ we have a commutative diagram
		\[
		\begin{tikzcd}[ampersand replacement=\&]
		\ho \left( \CC^{\dagger}_{d^{\prime}\upslash}\right)[W^{-1}] \arrow[r, "T^{\alpha}_{d^{\prime}}"] \arrow[d] \& \mathbb{A}(F(d^{\prime}),a)  \arrow[d] \\
		\ho \left( \CC^{\dagger}_{d\upslash}\right)[W^{-1}] \arrow[r,"T^{\alpha}_d"]  \&  \mathbb{A}(F(d),a)
		\end{tikzcd}
		\]
		where the  vertical morphisms are induced by precomposition by $w$ (resp. $F(w)$). 
		
		Our aim is to produce an inverse to the restriction functor
		\[
		\func{f_{!}:\mathbb{A}^{\dagger}_{f^{*}F\upslash}  \to \mathbb{A}^{\dagger}_{F\upslash}  }.
		\]
		We define the action of $f_{!}$ on a marked cone $\set{\alpha_c}_{c \in \CC}$ by the formula
		\[
		f_{!}(\alpha)_d= T^{\alpha}_d(\gamma_d).
		\]
		Given a morphism $\func{w: d \to d^{\prime}}$ we consider the following diagram in $\ho \left( \CC^{\dagger}_{d\upslash}\right)[W^{-1}]$ 
		\[
		\begin{tikzcd}[ampersand replacement=\&]
		d \arrow[d,"\gamma_d"] \arrow[r,"w"] \& d^{\prime} \arrow[d,"\gamma_{d^{\prime}}"]  \\
		f(c_d) \arrow[r,dotted,"\gamma_{d,d^{\prime}}"]  \& f(c_{d'})
		\end{tikzcd}
		\]
		and observe that the dotted arrow exists and is unique by our hypothesis. It is worth noting that if $w$ were marked, we would obtain  a morphism between  initial  objects and hence an  isomorphism. Now we define
		\[
		f_{!}(\alpha)_{w}= T^{\alpha}_d(\gamma_{d,d^{\prime}}).
		\]
		It is straightforward to check that conditions  of \autoref{def:markedcone} are satisfied. Thus $f_{!}$ is well defined on objects. Given $\func{\theta: a \to a'}$ and a morphism of marked cones $\set{\func{\epsilon_c:\theta \circ \alpha_c \nat \beta_c}}_{c \in \CC}$ we can produce  natural  transformation of functors $\func{\eta_d:(\theta*T)^{\alpha}_{d} \nat T^{\beta}_{d}}$ for every $d \in \DD$. Here $(\theta*T)^{\alpha}_{d}$ denotes the functor induced by postcomposition with $\theta$. This shows  that we can  define $f_{!}(\theta)_d=\eta_d(\gamma_d)$. After  some routine verifications we clearly see that the assignment on morphisms is functorial. The corresponding definition of the action $f_{!}$ on 2-morphisms is analogous.
		
		It is an straightforward exercise to check that the functors
		\[
		f^*:\mathbb{A}^{\dagger}_{F\upslash} \llra  \mathbb{A}^{\dagger}_{f^{*}F\upslash}: f_{!} 
		\]
		induce equivalences of categories upon  passage to fibers. Using \autoref{rem:scaledcocartesian} we see that $f^{*}$ is an equivalence of scaled coCartesian fibrations which in turn implies that the natural transformation $\func{N^{\dagger}_{F} \nat N^{\dagger}_{f^{*}F}}$ is a levelwise weak equivalence. \qedhere
	\end{implications}
\end{proof}

\begin{remark}
	The criteria of \autoref{thm:decatAdag} are quite sensitive to the choice of marking --- so much so, in fact, that operations which do not change the $\infty$-categorical localization (closing under 2-out-of-3, for example), can substantially alter marked cofinality. \autoref{exam:adjiscolim} provides an excellent demonstration of this characteristic --- if we take the marked colimit of $\func{T:[2]\to \Cat}$ where $[2]$ is equipped with the marking $[2]^\diamond$ from \autoref{exam:adjiscolim}, the marked colimit is the category $\mathcal{C}$. However, if we take the marking $[2]^\sharp$, the marked colimit is the pseudocolimit, and is thus equivalent to $\mathcal{D}$. This observation is also borne out by \autoref{thm:decatAdag} as applied to $\func{{[2]}^\diamond\to {[2]}^\sharp}$. The category $\on{ho}([2]^\diamond_{1/})[\mathcal{W}^{-1}]$ is isomorphic to $[1]$. However, the marked morphism $1\to 2$ in $[2]^\sharp$ corresponds to $1\in [1]$, and thus is not initial, so $\func{{[2]}^\diamond\to {[2]}^\sharp}$ is not marked cofinal.
\end{remark}

\begin{remark}
	In an earlier version of this paper, we deployed another definition of a marked colimit in a 2-category. Under this definition, we asserted that, in analogy to the 1-categorical case, a marked colimit is a 2-initial object in an appropriate 2-category of marked cocones. The recent paper \cite{Clingman}, convinced us that, while there is a connection between this definition and the one given in the current version of the paper, the current version is the correct notion of a marked colimit. A representation of the functor $N_F^\dagger$ determines a 2-initial object in the 2-category of marked cocones, however, the existence of such a 2-initial object does \emph{not} imply that the functor in question is representable. Our ongoing work on the cofinality conjecture had also led us to the conclusion that the approach based on representable functors lends itself far better to the $\infty$-categorical case.  
\end{remark}

\subsection{The cofinality conjecture}

We now turn to our main purpose in discussing marked 2-colimits: the cofinality conjecture. Before this, however, we offer some intermediate conjectures which suggest the broad strokes of a larger theory. In this section we will make use of some concepts which have not yet been rigorously defined, for instance lax and oplax slices of an $(\infty,2)$-functor.

Or first conjecture concerns the background notion of $(\infty,2)$-colimit:
\begin{conjecture}\label{conj:infty2colimsexist}
	There is a theory of marked $(\infty,2)$-colimits in any $(\infty,2)$-category categorifying the strict 2-categorical theory. Such a theory should take as input a functor $\func{F:X\to \mathbfcal{C}}$, where $X$ is a marked simplicial set, and yield a marked $(\infty,2)$-colimit cone over $F$ as output.  
\end{conjecture}

Examining the 2-categorical definition more closely, we also expect certain special cases to arise:

\begin{conjecture}
	In the case of the $\infty$-bicategory $\mathfrak{C}\!\on{at}_{\infty}$: 
	\begin{itemize}
		\item The marked $(\infty,2)$-colimit of a functor $\func{F:X^\flat\to \mathfrak{C}\!\on{at}_{\infty}}$ coincides with the lax $\infty$-colimit of the underlying functor of $F$.
		\item The marked $(\infty,2)$-colimit  of a functor $\func{F:X^\sharp\to \Cat_\infty \to \mathfrak{C}\!\on{at}_{\infty}}$ coincides with the $(\infty,1)$-colimit. 
	\end{itemize}
\end{conjecture}

We will now assume that \autoref{conj:infty2colimsexist} holds. There is good reason to believe this should be true --- using techniques similar to those of \cite[Notation 4.1.5]{Lurie_Goodwillie}, one can write down an analogue of the category of marked cocones. The difficulties that then follow are technical: proving fibrancy, identifying functoriality, and relating various dual constructions (the last is somewhat complicated by the fact that the Duskin 2-nerve forces us to fix a convention on the direction of compositors, and does not play well with dualizing 2-morphisms). 

Based both on Theorem A\textsuperscript{$\dagger$} and on the relation of the decategorified criterion to marked 2-colimits, we then propose the following conjecture: 

\begin{conjecture}[The cofinality conjecture]
	Precomposition with a marked $(\infty,2)$-functor 
	\[
	\func{F: \mathbfcal{C}^\dagger \to \mathbfcal{D}^\dagger}
	\] preserves marked $(\infty,2)$-colimits if and only if the following two conditions are satisfied for every object $d\in \mathbfcal{D}$:
	\begin{itemize}
		\item There is an object $c\in \mathbfcal{C}$ and a morphism \begin{tikzcd}
			d\arrow[r] & F(c)
		\end{tikzcd}
		initial in the $(\infty,1)$-localizations of the slices $(\mathbfcal{C}_{d\upslash})^\dagger$ and $(\mathbfcal{D}_{d\upslash})^\dagger$.
		\item Every marked morphism is initial in the $(\infty,1)$-localization of the marked $(\infty,2)$ slice category $(\mathbfcal{C}_{d\upslash})^\dagger$. 
		\item For any marked morphism \begin{tikzcd}
			f: d \arrow[r,circled] & b
		\end{tikzcd}
		the induced functor $\func{f^\ast: L_{\mathcal{W}}((\mathbfcal{C}_{b\upslash})^\dagger)\to L_{\mathcal{W}}((\mathbfcal{C}_{d\upslash})^\dagger)}$ preserves initial objects.
	\end{itemize}
\end{conjecture}

\begin{remark}
	Establishing the appropriate definitions and proving the cofinality conjecture is the subject of ongoing work. In the upcoming paper \cite{AG}, the first author proves the cofinality conjecture for functors $\func{F:\mathcal{C}^\dagger\to \mathfrak{C}\!\on{at}_{\infty}}$, where $\mathcal{C}^\dagger$ is a marked $(\infty,1)$-category. 
\end{remark}

%

\end{document}